\documentclass[11pts]{amsart}
\usepackage{amssymb}
\usepackage{epsfig}
\usepackage{amscd}
\usepackage{graphicx}

\newtheorem{theorem}{Theorem}[section]
\newtheorem{prop}[theorem]{Proposition}
\newtheorem{cor}[theorem]{Corollary}
\newtheorem{lem}[theorem]{Lemma}
\newtheorem{rem}[theorem]{Remark}

\newtheorem{rems}[theorem]{Remarks}
\def\ra{\rightarrow}
\def\o{\otimes}
\def\ot{\otimes}
\def\b{\beta}
\def\s{\sigma}
\def\a{\alpha}
\def\c{\gamma}
\def\G{\Gamma}
\def\mc{\mathcal}
\def\lan{\langle}
\def\ran{\rangle}

\def\l{\lambda}

\def\ve{\varepsilon}

\def\Kz{Kri\v{z} }
\def\S3{\mathcal{S}_3}

\def\pr{\rm pr}
\def\t{\tau}

\begin{document}
\title{Cohomology of 3-points configuration spaces of complex projective spaces}
\thanks{2010 AMS Classification Primary: 55R80, 20C30; Secondary: 55P62,
13A50.\\
\indent This research is partially supported by Higher Education Commission,
Pakistan.}
\author{Samia Ashraf$^{1}$, Barbu Berceanu$^{2}$}
\address{$^{1}$Abdus Salam School of Mathematical Sciences, GC University,
Lahore-Pakistan.}
\email{samia.ashraf@yahoo.com}
\address{$^{2}$Abdus Salam School of Mathematical Sciences, GC University,
Lahore-Pakistan, and Institute of Mathematics Simion Stoilow,
Bucharest-Romania (Permanent address).}
\email{Barbu.Berceanu@imar.ro}
\maketitle

\begin{abstract}
We compute the Betti numbers and describe the cohomology algebras
of the ordered and unordered configuration spaces of three points in complex projective spaces,
including the infinite dimensional case. We
also compute these invariants for the configuration spaces of
three collinear and non-collinear points.
\end{abstract}

\pagestyle{myheadings} \markboth{\centerline {\scriptsize SAMIA
ASHRAF,\,\,\,\,BARBU BERCEANU}}{\centerline {\scriptsize 3-points
configuration spaces}}

\textwidth=13cm \pagestyle{myheadings}

\section{Introduction}\label{section1}

The ordered configuration space of $n$ points $F(X,n)$ of a
topological space $X$ is defined as
$$F(X,n)=\{(x_1,x_2,\ldots,x_n)\in X^n \mid x_i\neq x_j \mbox{ for } i\neq j\}.$$
The symmetric group on $n$ letters acts freely on this space and
its orbit space is the unordered configuration space, denoted by
$C(X,n)$. The induced action on the cohomology algebra (with
rational coefficients) $H^*(F(X,n))$ has as the invariant part the
rational cohomology of $C(X,n)$.

For $X$ a smooth complex projective variety, I. Kri\v{z} \cite{K}
constructed a rational model $E(X,n)$ for $F(X,n)$, a simplified
version of the Fulton-MacPherson model \cite{FM}. Let us remind
the construction of Kri\v{z}. We denote by $p_i^*:H^*(X)\ra
H^*(X^n)$ and $p_{ij}^*:H^*(X^2)\ra H^*(X^n)$ (for $i\neq j$) the
pullbacks of the obvious projections and by $m$ the complex
dimension of $X$. The model $E(X,n)$ is defined as follows: as an
algebra, $E(X,n)$ is isomorphic to the exterior algebra with
generators $G_{ij},\,1\leq i,j\leq n$ (of degree $|G_{ij}|=2m-1$)
and coefficients in $H^*(X)^{\o n}$ modulo the relations
$$\begin{array}{llll}
  G_{ji}         & = & G_{ij},                    &  \\
  p_j^*(x)G_{ij} & = & p_i^*(x)G_{ij},            & (i< j),\, x\in H^*(X), \\
  G_{ik}G_{jk}   & = & G_{ij}G_{jk}-G_{ij}G_{ik}, & (i<j<k).
\end{array}$$
The differential $d$ is given by $d|_{H^*(X)^{\o n}}=0$ and
$d(G_{ij})=p_{ij}^*(\Delta)$, where $\Delta= w\o1+\ldots+1\o w\in
H^*(X)\o H^*(X)$ denotes the class of the diagonal and $w\in
H^{2m}(X)$ is the fundamental class. This model is a differential
bigraded algebra
$E(X,n)=\mathop{\bigoplus}\limits_{k,q}E_q^k(X,n)$: the lower
degree $q$ (called the exterior degree) is given by the number of
exterior generators $G_{ij}$, and the upper degree $k$ is given by
the total degree. The multiplication is homogeneous and the
differential has bidegree ${+1}\choose{-1}$
$$E_q^k\o E_{q'}^{k'}\ra E_{q+q'}^{k+k'},\,\,\,d:E_q^k\ra E_{q-1}^{k+1},$$
therefore the cohomology algebra of the configuration space,
$H_*^*=H_*^*(F(X,n))$, is a bigraded algebra and we will compute
the two variable Poincar\'{e} polynomial
$$ P_{F(X,n)}(t,s)=\sum_{k,q\geq 0}(\dim H_q^k)t^ks^q. $$
The \emph{canonical basis} for the Kri\v{z} model is given by
products of exterior elements
$G_{I_*J_*}=G_{i_1j_1}G_{i_2j_2}\ldots G_{i_qj_q}$, where
$I_*=(i_1,\ldots,i_q),\,J_*=(j_1,\ldots,j_q)$, satisfying
$i_a<j_a$ $(a=1,2,\ldots,q)$ and $j_1<j_2<\ldots<j_q$, with
scalars $x_*=x_1\o x_2\o \ldots \o x_n$ (the factors belong to a fixed basis of $H^*(X)$), and
$x_j=1$ if $j\in J_*$, see \cite{B}, \cite{BMP} or \cite{AAB} for more details).

The symmetric group $\mathcal{S}_n$ acts on each bigraded component $E_q^k(X,n)$ by
permuting the factors in $H^*(X^n)=H^{*\otimes n}$ and changing the indices of the exterior
generators (see \cite{FM} and \cite{K}): for an arbitrary permutation $\sigma\in \mathcal{S}_n$
$$ \sigma\big( p_1^*(x_1)\dots p_n^*(x_n)G_{I_*J_*}\big)=p_{\sigma(1)}^*(x_1)\dots
p_{\sigma(n)}^*(x_n)G_{\sigma(I_*)\sigma(J_*)}. $$
The differential is equivariant (see \cite{AAB}), and therefore we can split the
differential algebra $(E_*^*,d)$ into isotypical parts
corresponding to partitions $\lambda$ of $n$, $\lambda\vdash n$,
$$ (E_*^*(X,n),d)\cong\bigoplus_{\lambda\vdash n}(E_*^*(X,n)(\lambda),d_{\lambda}).$$

In \cite{FTh} Felix and Thomas proved that for even dimensional closed manifold $M$
the rational Betti numbers of $C(M,n)$ are
determined by the graded cohomology algebra $H^*(M;\mathbb{Q})$, and, as an application, they
computed the Betti numbers of $C(\mathbb{C}P^3,3)$. A presentation of the rational cohomology of
$C(\mathbb{C}P^2,n)$ appears in \cite{FT} by analyzing the $\mc{S}_n$-invariant part of the spectral sequence
converging to $H^*(F(M,n))$, introduced by Cohen and Taylor \cite{CT}; see also \cite{FOT}. In \cite{So} \Kz model
and some natural fibrations are used to compute $H^*(F(\mathbb{C}P^m,2))$ for an arbitrary finite $m$
and $H^*(F(\mathbb{C}P^m,3))$ for $m\leq4$. In \cite{FZ} Feichtner and Ziegler described the cohomology algebra of
$F(\mathbb{C}P^1,n)$ with integer coefficients. The symmetric structure of $H^*(F(\mathbb{C}P^1,n))$
 was considered in \cite{AAB}.

In this paper we analyze the cohomology of various configuration
spaces of three points in complex projective spaces $\mathbb{C}P^m$ for $m\geq2$,
including the infinite dimensional space, using the
representation theory of the Kri\v{z} model and some associated
spectral sequences.

In Section 2 we describe the symmetric structure of the \Kz model
$E_*^*=E_*^*(\mathbb{C}P^m,3)$. The irreducible modules corresponding to the
partition $\l$ will be denoted by $V(\l)$, as in \cite{FH}. The multiplicities of the
irreducible $\S3$-modules in $E_q^k$ are given in terms of
partitions and bounded partitions of an integer into three parts,
see Proposition \ref{lem2.1}. Explicit formulae for these numbers are given in
the Appendix. In the stable cases, $m\geq5$, the pairs $(k,q)$ with
non-zero $E_q^k(\lambda)$ and those
with non-zero $H_q^k(\lambda)$ are given in the next diagrams (the former
pairs are marked with small circles and the latter with bullets; the dots mark the
places where $E_q^k\neq 0$ and $E_q^k(\lambda)=0$).

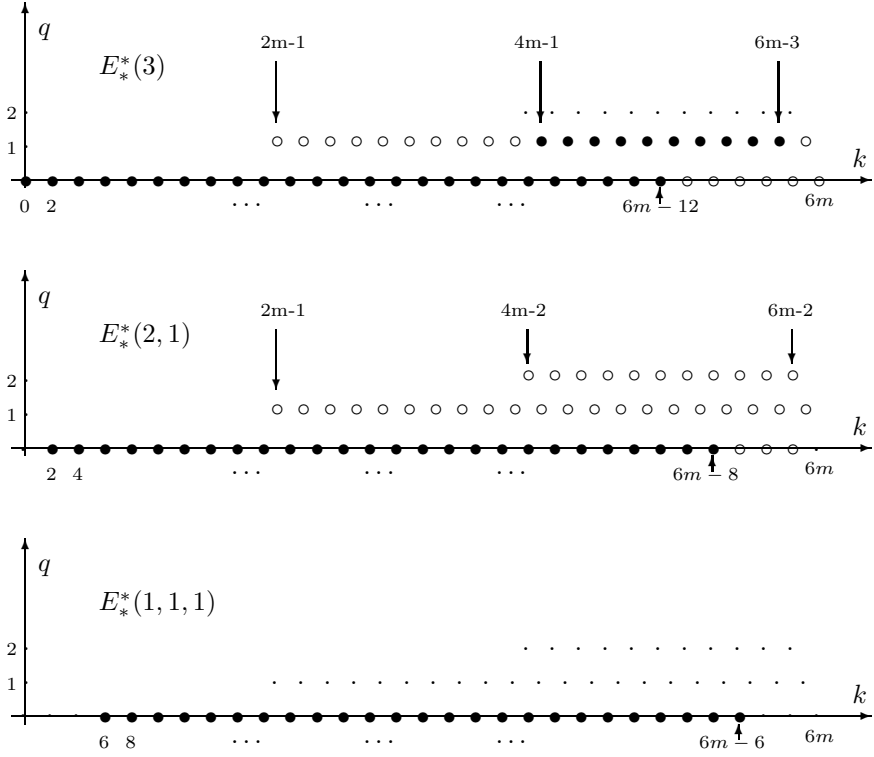
\begin{picture}(360,90)
\put(117,55){\vector(0,-1){23}} \put(217,55){\vector(0,-1){23}}
\put(17,10){\vector(1,0){325}} \put(22,7){\vector(0,1){70}}
\put(27,65){$q$}    \put(335,15){$k$}
\put(307,55){\vector(0,-1){23}} \multiput(21,22)(0,13){2}{$. $}
\put(50,50){$E_*^*(3)$} \multiput(115,22)(10,0){10}{$\circ$}
\multiput(215,22)(10,0){10}{$\bullet$}
\multiput(210,35)(10,0){11}{$.$}         \put(315,22){$\circ$}
\multiput(20,7)(10,0){25}{$\bullet$}
\multiput(270,7)(10,0){6}{$\circ$}
\multiput(100,0)(50,0){3}{$\ldots$} \put(262,1){\vector(0,1){6}}
\begin{scriptsize}    \put(20,-2){$0$}   \put(30,-2){$2$} \put(111,60){2m-1}
\put(248,-2){$6m-12$} \put(317,0){$6m$}  \put(15,20){$1$} \put(207,60){4m-1}
\put(15,33){$2$}      \put(298,60){6m-3}
\end{scriptsize}
\end{picture}

\bigskip

\begin{picture}(360,90)
\put(117,55){\vector(0,-1){23}} \put(212,55){\vector(0,-1){13}}
\put(312,55){\vector(0,-1){13}} \put(50,50){$E_*^*(2,1)$}
\put(17,10){\vector(1,0){325}}  \put(22,7){\vector(0,1){70}}
\put(27,65){$q$}                        \put(335,15){$k$}
\multiput(20,9)(300,0){2}{$. $} \put(282,1){\vector(0,1){6}}
\multiput(115,22)(10,0){21}{$\circ$}
\multiput(21,22)(0,13){2}{$.$}
\multiput(210,35)(10,0){11}{$\circ$}
\multiput(30,7)(10,0){26}{$\bullet$}
\multiput(290,7)(10,0){3}{$\circ$}
\multiput(100,0)(50,0){3}{$\ldots$}
\begin{scriptsize}    \put(30,-2){$2$}  \put(40,-2){$4$}  \put(111,60){2m-1}
\put(267,-2){$6m-8$}  \put(317,0){$6m$} \put(15,20){$1$}  \put(202,60){4m-2}
\put(15,33){$2$}      \put(303,60){6m-2}
\end{scriptsize}
\end{picture}

\bigskip
\begin{picture}(360,90)
\put(17,10){\vector(1,0){325}} \put(22,7){\vector(0,1){70}}
\put(27,65){$q$} \put(335,15){$k$} \multiput(20,9.5)(10,0){31}{$. $}
\put(50,50){$E_*^*(1,1,1)$} \multiput(115,22)(10,0){21}{$.$}
\multiput(210,35)(10,0){11}{$.$}
\multiput(50,7)(10,0){25}{$\bullet$} \put(292,1){\vector(0,1){6}}
\multiput(100,0)(50,0){3}{$\ldots$} \multiput(21,22)(0,13){2}{$.$}
\begin{scriptsize}    \put(50,-2){$6$}  \put(60,-2){$8$}
\put(277,-2){$6m-6$}  \put(317,0){$6m$} \put(15,20){$1$}
\put(15,33){$2$}
\end{scriptsize}
\end{picture}

\bigskip
In Section \ref{section3} we compute the Betti numbers of each
isotypical part of $E_*^*(\mathbb{C}P^m,3)$. Summing up their
Poincar\'{e} polynomials we obtain
\begin{theorem}\label{thm1}
The two variable Poincar\'{e} polynomial for the ordered configuration space $F(\mathbb{C}P^m,3)$ is given by:
$$P_{F(\mathbb{C}P^m,3)}(t,s)= C_{m-1}^3(t)+st^{4m-1}C_{m-1}(t).$$
\end{theorem}
In this theorem and the next ones we use the notation
$$r_m(x,y)=x^m+x^{m-1}y+\ldots+y^m,C_m(t)=r_m(1,t^2)=1+t^2+\ldots+t^{2m}.$$
In the Appendix one can find the proof of following:
\begin{cor}\label{cor1}
The sequences of even and odd Betti numbers of the configuration
space $F(\mathbb{C}P^m,3)$ and also the sequences of its
isotypic parts are unimodal:
$$\begin{array}{lllll}
\b_0\leq\b_2\leq  & \ldots & \leq\b_{2i}\quad\geq\b_{2i+2}\geq &
                    \ldots & \geq\b_{6m-2}\geq\b_{6m},\\
\b_1\leq\b_3\leq  & \ldots & \leq\b_{2j-1}\geq\b_{2j+1}\geq &
                    \ldots & \geq\b_{6m-3}\geq\b_{6m-1}.
\end{array}$$
\end{cor}
In Section 4 the multiplicative structure of the cohomology
algebras $H^*(F(\mathbb{C}P^m,3))$ and $H^*(C(\mathbb{C}P^m,3))$ is analyzed.
\begin{theorem}\label{thm3}
The cohomology algebra $H^*(F(\mathbb{C}P^m,3))$ has the presentation (as a graded commutative algebra):
$$\lan\a_1,\a_2,\a_3,\eta\mid \a_1^{m+1},r_m(\a_i,\a_j), (\a_i-\a_j)\eta,\a_1^m\eta\ran$$
where $|\a_i|=2$, $|\eta|=4m-1.$
\end{theorem}
From the Newton's version of the fundamental theorem of symmetric polynomials,
there are uniquely defined polynomials $P_k\in\mathbb{Q}[X_1,X_2,X_3]$ such that
$$X_1^k+X_2^k+X_3^k=P_k(X_1+X_2+X_3,X_1^2+X_2^2+X_3^2,X_1^3+X_2^3+X_3^3),$$
and these are used in Theorem \ref{thm2} and in Section \ref{section4}.
\begin{theorem}\label{thm2}
The cohomology algebra $H^*(C(\mathbb{C}P^m,3))$ has the
presentation (as a graded commutative algebra):
$$\left\lan\begin{array}{c}
            \t_1,\t_2,\t_3, \\
                      \\
            \eta
           \end{array}\right.
\left|\begin{array}{c}
            P_{m+1},P_{m+2},P_{m+3},(\t_1^2-3\t_2)\eta, (\t_1\t_2-3\t_3)\eta, P_m\eta, \\
            (m+1)P_m-\mathop{\sum}\limits_{i=0}^m P_iP_{m-i},\mathop{\sum}\limits_{i=1}^{m-1} P_iP_{m-i},
                                \mathop{\sum}\limits_{i=2}^{m-2} P_iP_{m-i}
                             \end{array}\right\ran$$
where $|\t_i|=2i$ $(i=1,2,3)$, $|\eta|=4m-1$ and $P_k=P_k(\t_1,\t_2,\t_3)$.
\end{theorem}

In the next Section we describe the ``stable cohomology" of the
configuration spaces: as the limit
$\lim_mF(\mathbb{C}P^m,n)=F(\mathbb{C}P^{\infty},n)$ has the
homotopy type of $(\mathbb{C}P^{\infty})^n$, see Theorem \ref{theorem5.1},
we obtain
\begin{theorem}\label{thm4}
The cohomology algebras of the ordered and unordered configuration
spaces of the infinite complex projective space are given by
$$ \begin{array}{lll}
   H^*(F(\mathbb{C}P^{\infty},n)) & \cong & \mathbb{Q}[x_1,x_2,\ldots,x_n],\\
   H^*(C(\mathbb{C}P^{\infty},n)) & \cong &
   \mathbb{Q}[\s_1,\s_2,\ldots,\s_n],
\end{array} $$
where $|x_i|=2$ and $|\s_i|=2i$.
\end{theorem}
\noindent and the corollary
\begin{cor}\label{cor2}
The inclusion $\mathbb{C}P^m\hookrightarrow\mathbb{C}P^{\infty}$ induces an isomorphism
$$ H^k(F(\mathbb{C}P^{\infty},n))\cong H^k(F(\mathbb{C}P^m,n)) $$
for $0\leq k\leq 2m-1$.
\end{cor}

In Section 6 we compute the Betti numbers of the spaces of configurations of three collinear points
and configurations of three non-collinear points (see \cite{BP}
for the fundamental groups of these spaces):
$$ \begin{array}{lll}
     F_c(\mathbb{C}P^m,3) & = & \{(p,q,r)\in F(\mathbb{C}P^m,3)\mid p,q,r \mbox{ are
     collinear}\},\\
     F_{nc}(\mathbb{C}P^m,3) & = & F(\mathbb{C}P^m,3)\setminus
     F_c(\mathbb{C}P^m,3),\\
\end{array} $$
\begin{theorem}\label{thm5}
The Poincar\'{e} series of the space of configurations of three collinear points in $\mathbb{C}P^m$ is given by:
$$P_{F_c(\mathbb{C}P^m,3)}(t)=(1+t^{2m+1})C_{m-1}(t)=\frac{(1+t^{2m+1})(1-t^{2m})}{1-t^2},$$
in the finite case. If $m$ is infinite we have
$$P_{F_c(\mathbb{C}P^{\infty},3)}(t)=\frac{1}{1-t^2}.$$
\end{theorem}
\begin{theorem}\label{thm6}
The Poincar\'{e} series of the ordered configuration space
of three non-collinear points in complex projective space is given as:
$$P_{F_{nc}(\mathbb{C}P^m,3)}(t)=C_{m-2}(t)\cdot C_{m-1}(t)\cdot C_m(t)=\frac{(1-t^{2m-2})(1-t^{2m})(1-t^{2m+2})}{(1-t^2)^3},$$
in the finite case. If $m$ is infinite we have
$$P_{F_{nc}(\mathbb{C}P^\infty,3)}(t)=\frac{1}{(1-t^2)^3}.$$
\end{theorem}

\section{The $\mc{S}_3$-decomposition of $E_*^*(\mathbb{C}P^m,3)$}
\label{section2}

The equation $a+b+c=k$ ($a,b,c$ are non-negative integers) has
three types of solutions: the first type when $a,b,c$ are pairwise
distinct, the second type when the set $\{a,b,c\}$ has two
elements, and the third type when $a=b=c$ (if $3\mid k$). The symmetric group
$\mc{S}_3$ acts on the set of triplets $(a,b,c)$ and linearizing
this action we find the next $\mc{S}_3$-modules corresponding to
the three types:
$$\begin{array}{lll}
  \mathbb{Q}[\mc{S}_3]\lan a,b,c\ran & \cong & V(3)\oplus 2V(2,1)\oplus V(1,1,1), \\
  \mathbb{Q}[\mc{S}_3]\lan a,a,b\ran & \cong & V(3)\oplus V(2,1), \\
  \mathbb{Q}[\mc{S}_3]\lan a,a,a\ran & \cong & V(3).
\end{array}
$$
We denote by $P_3(k)$ the cardinality of partitions of $k$ into
three non-negative parts
$$ P_3(k)=\mathrm{card}\{(a,b,c)\mid  a\geq b\geq c\geq0, a+b+c=k\} $$
and also by $P_{3,\leq m}(k)$ the cardinality of bounded
partitions
$$P_{3,\leq m}(k)=\mathrm{card}\{(a,b,c)\mid m\geq a\geq b\geq c\geq 0, a+b+c=k\}.$$
See the Appendix for the computation of these numbers and some of
their elementary properties.

In the next statement we will use the notation $r_3(k)=1$ if $3|k$ and 0
otherwise; the component $E_q^k(\mathbb{C}P^m,3)$ is denoted shortly $E_q^k$. From
now on, $H^*$ will denote the cohomology algebra
$H^*(\mathbb{C}P^m;\mathbb{Q})\cong\mathbb{Q}[x]\diagup\lan x^{m+1}\ran$, where $|x|=2$; the
canonical basis is $\{ 1,x,\ldots,x^m\}$ and $\Delta=x^m\o 1+x^{m-1}\o x+\ldots +1\o x^m $.
\begin{prop}\label{lem2.1}
The non-zero bigraded components $E_q^k$ are decomposed into
irreducible $\mc{S}_3$-modules with multiplicities from the following table:

\begin{center}\begin{tabular}{|c|c|c|c|}
  \hline
                                       & $V(3)$ & $V(2,1)$ & $V(1,1,1)$ \\ \hline
  $\begin{array}{cl}
  E_2^{2(2m-1)+2k} \cong E_2^{2(3m-1)-2k} & (0\leq 2k\leq m) \\
  E_1^{(2m-1)+2k}  \cong E_1^{(6m-1)-2k}  & (0\leq k\leq m)  \\
  \quad\,\,\,\,\, E_0^{2k}         \cong E_0^{6m-2k}      & (0\leq 2k\leq3m)
  \end{array}$
& $\begin{array}{c}
      -       \\
      k+1     \\
      \mu_3(k)
\end{array}$
& $\begin{array}{c}
      1       \\
      k+1     \\
      \mu_{2,1}(k)
\end{array}$
& $\begin{array}{c}
      -       \\
      -       \\
      \mu_{1,1,1}(k)
\end{array}$\\  \hline
\end{tabular}
\end{center}
where
$$ \begin{array}{lll}
  \mu_3(k)       & = & P_{3,\leq m}(k), \\
  \mu_{2,1}(k)   & = & 2P_{3,\leq m}(k)-(\lfloor\frac{k}{2}\rfloor+1+r_3(k)-max\{0,\lceil\frac{k-m}{2}\rceil\}) \\
                 & = & P_{3,\leq m}(k)+P_{3,\leq m-2}(k-3)-r_3(k),   \\
  \mu_{1,1,1}(k) & = & P_{3,\leq m}(k)-(\lfloor\frac{k}{2}\rfloor+1-max\{0,\lceil\frac{k-m}{2}\rceil\})\\
                 & = & P_{3,\leq m-2}(k-3).
\end{array}
$$
\end{prop}
\begin{proof}
All the cohomology classes in $H^*$ are of even
degree, hence $E_q^k=0$ if $k+q\equiv1\pmod 2$. For large values
of $k$ we use the $\mc{S}_3$-equivariant Poincar\'{e} duality, see
\cite{AAB}.

0) The action of $\mc{S}_3$ on $E_0^{2k}$ is equivalent with the action on the set of non negative solutions
of the equation $a+b+c=k$, $a,b,c\leq m $ and the decompositions of the three types were given above. The
multiplicity $\mu_{1,1,1}(k)$ is given by the number of solutions of the equation
$$a+b+c=k,\,\,\,m\geq a>b>c\geq 0,$$
and this number equals the number $P_{3,\leq m}(k)$ minus the
number of solutions of the equation
$$2a+b=k,\,\,\,\,m\geq a,b\geq 0$$
and this gives the interval
$\max\{0,\lceil\frac{k-m}{2}\rceil\}\leq a\leq
\lfloor\frac{k}{2}\rfloor$, (here $k\leq \frac{3m}{2}$). The
second formula is a consequence of the bijection
$$ \begin{array}{ccc}
  \left\{(a,b,c)\Big|\begin{array}{l}  a+b+c=k,\\
                             m\geq a>b>c\geq 0
            \end{array}  \right\}    & \approx &
  \left\{(\a,\b,\c)\Big|\begin{array}{l} \a+\b+\c=k-3,\\
                               m-2\geq\a\geq\b\geq\c\geq 0
            \end{array}   \right\}     \\
  (a,b,c) & \leftrightarrow & (a-2,b-1,c).
\end{array} $$
The multiplicity $\mu_{2,1}(k)$ is given by the total number of
solutions of the first two types, $P_{3,\leq m}(k)-r_3(k)$, plus the
number of solutions of the first type, $\mu_{1,1,1}(k)$.

1) The elements of the canonical basis $E_1^{(2m-1)+2k}$ correspond to the marked graphs
$(\Gamma_{(2,1)},(x^{h'},x^{h''}))$ such that $h'+h''=k$, where
$0\leq k\leq m$ (see \cite{AAB} for a marked version of the
graphs introduced in \cite{LS}). The possible marks $(h',h'')$
are $(0,k),(1,k-1),\ldots,(k,0)$ and the corresponding marked
graphs are:

\begin{picture}(0,60)
\multiput(0,30)(30,0){3}{$\centerdot$}  \put(1,31){\line(1,0){30}}
\put(-2,22){$1$}       \put(28,22){$2$}          \put(58,22){$3$}
\put(15,35){$x^{h_1}$} \put(60,35){$x^{h_3}$}    \put(90,30){$,$}
\multiput(120,30)(30,0){3}{$\centerdot$}
\put(151,31){\line(1,0){30}}  \put(118,22){$1$}  \put(148,22){$2$}
\put(178,22){$3$}             \put(117,35){$x^{h_1}$}
\put(162,35){$x^{h_2}$}       \put(210,30){$,$}  \put(240,0){
\multiput(0,30)(30,0){3}{$\centerdot$} \put(1,31){\line(1,0){30}}
\put(-2,22){$1$} \put(28,22){$3$}      \put(58,22){$2$}
\put(15,35){$x^{h_1}$}                 \put(60,35){$x^{h_2}$} }
\end{picture}
The transposition $(12)$ preserves the first marked graph and the
3-cycles have no contribution to the trace, hence the character of
this representation gives the decomposition $V(3)\oplus V(2,1)$.

2) For $E_2^{2(2m-1)+2k}$ we have a unique type of associated marked graphs,
$(\Gamma_{(3)},x^{h_1})$, giving the irreducible module $V(2,1)$:

\begin{picture}(30,70)
\put(0,30){$(\Gamma_{(3)},x^{h_1}):$}
\multiput(90,30)(30,0){3}{$\centerdot$}
\multiput(91,31)(30,0){2}{\line(1,0){30}} \put(88,22){$1$}
\put(118,22){$2$} \put(148,22){$3$} \put(117,40){$x^{h_1}$}
\put(180,30){$,$} \put(120,0){
\multiput(90,30)(30,0){3}{$\centerdot$}
\multiput(91,31)(30,0){2}{\line(1,0){30}} \put(88,22){$2$}
\put(118,22){$1$} \put(148,22){$3$} \put(117,40){$x^{h_1}$} }
\end{picture}
\end{proof}

\section{Computing the differentials}
\label{section3}

For three points in $\mathbb{C}P^m$ we have the decomposition
$$ (E_*^*,d)=(E_*^*(3),d_3)\oplus (E_*^*(2,1),d_{2,1})\oplus(E_*^*(1,1,1),d_{1,1,1})$$
corresponding to the three partitions of 3. We analyze the
properties of the differential in each of these three components.
\begin{lem}\label{lem31}
For $0\leq k\leq m-1$ the differential $d_3$ is injective
$$ d:E_{1}^{(2m-1)+2k}(3)\rightarrowtail E_{0}^{2m+2k}(3). $$
\end{lem}
\begin{proof}
We introduce filtrations on both $E_{1}^{(2m-1)+2k}(3)$ and $E_{0}^{2m+2k}(3)$. For a triple $\a=(a,b,c)$ with
$|\a|= a+b+c=m+k$ and $m\geq a\geq b\geq c\geq 0$ we denote by $v_{\a}$ the element of $E_0^{2m+2k}$
$$v_{\a}=\mathop {\sum}\limits_{\pi\in \S3} \pi(x^{a}\otimes x^{b}\otimes x^{c}).$$
It is obvious that $\{v_{\a}\mid |\a|=m+k\}$ is a basis of $E_0^{2m+2k}(3).$ Using the lexicographic order on
$\a$'s we define an increasing filtration on $E_0^{2m+2k}(3)$ by
$$\mc{T}_{\a}=\mathop{\bigoplus}\limits_{\b \leq \a}\mathbb{Q}\lan v_{\b}\ran.$$

We choose the basis of the module $E_{1}^{(2m-1)+2k}(3)$ given by:
$$ \begin{array}{ll}
  \sigma_{a}=& (x^{a}\otimes 1\otimes x^{k-a}+x^{k-a}\otimes 1\otimes x^a)G_{12}+ (x^{a}
              \otimes x^{k-a}\otimes 1+x^{k-a}\otimes x^a\otimes1)G_{13}  \\
             & +(x^{k-a}\otimes x^{a}\ot 1+x^{a}\otimes x^{k-a}\otimes 1)G_{23}\\
  \delta_{a}= &(x^{a}\otimes 1\otimes x^{k-a}-x^{k-a}\otimes 1\otimes x^a)G_{12}+ (x^{a}
               \otimes x^{k-a}\otimes 1-x^{k-a}\otimes x^a\otimes1)G_{13} \\
              & +(x^{k-a}\otimes x^{a}\ot 1-x^{a}\otimes x^{k-a}\otimes 1)G_{23},
\end{array}
$$
where $\lceil \frac{k}{2}\rceil\leq a \leq k.$

We define an increasing filtration on $E_{1}^{(2m-1)+2k}(3)$ in two steps:
$$\mathcal{D}_{a}=\mathop{\bigoplus}\limits_{b\leq a}\mathbb{Q}\lan\delta_{b}\ran \mbox{   and   }
\mathcal{G}_{a}=\mathcal{D}_k\bigoplus
(\mathop{\oplus}\limits_{b\leq a}\mathbb{Q}\lan\sigma_{b}\ran).$$
The leading monomial of $d(\delta_a)$ is given by the projection
of $d(\delta_a)$ on the monomials containing $x^m$ or $x^{m-1}$ on
the first position:
$$ \begin{array}{rcl}
  {\pr}_{(m,m-1)}(d(\delta_a)) & = & x^{m}\otimes x^{a}\otimes x^{k-a}+x^{m-1}\otimes x^{a+1}\otimes x^{k-a}- \\
              &   & -x^{m}\otimes x^{k-a}\otimes x^{a}- x^{m-1}\otimes x^{k-a+1}\otimes x^a+\\
              &   & +x^{m}\otimes x^{k-a}\otimes x^{a}+x^{m-1}\otimes x^{k-a}\otimes x^{a+1}- \\
              &   & -x^{m}\otimes x^{a}\otimes x^{k-a}-x^{m-1}\otimes x^{a}\otimes x^{k-a+1}=\\
              &  = & x^{m-1}\otimes x^{a+1}\otimes x^{k-a}+x^{m-1}\otimes x^{k-a}\otimes x^{a+1}-\\
              &   &  - x^{m-1}\otimes x^{k-a+1}\otimes x^a-x^{m-1}\otimes x^{a}\otimes x^{k-a+1}
\end{array}
$$
On the other hand,
$$ \begin{array}{rcl}
  {\pr}_{(m,m-1)}(d(\sigma_a)) & = & x^{m}\otimes x^{a}\otimes x^{k-a}+x^{m-1}\otimes x^{a+1}\otimes x^{k-a}+ \\
                             &   & +x^{m}\otimes x^{k-a}\otimes x^{a}+ x^{m-1}\otimes x^{k-a+1}\otimes x^a+\\
              &   & +x^{m}\otimes x^{k-a}\otimes x^{a}+x^{m-1}\otimes x^{k-a}\otimes x^{a+1}+ \\
              &   & +x^{m}\otimes x^{a}\otimes x^{k-a}+x^{m-1}\otimes x^{a}\otimes x^{k-a+1},\\
\end{array}
$$
hence $d(\sigma _a)$ has $x^m\otimes x^a\otimes x^{k-a}$ as the leading monomial. These computations
imply the injectivity of the map induced by the differential $d$ between the graded modules associated to
the filtrations $(\mc{D}_*,\mc{G}_*)$ and $\mc{T}_*$:
$$(\mathop{\oplus}\limits_{a=\lceil\frac{k}{2}\rceil}^{k}\mc{D}_{a}\diagup \mc{D}_{a-1})\bigoplus
(\mathop{\oplus}\limits_{a=\lceil\frac{k}{2}\rceil}^{k}\mc{G}_{a}\diagup \mc{G}_{a-1})\mathop{\rightarrowtail}
\limits^{d} \mathop{\bigoplus}\limits_{\a}\mc{T}_{\a}\diagup \mc{T}_{<\a}$$
(here $\mc{D}_{\lceil\frac{k}{2}\rceil-1}=0$ and $\mc{G}_{\lceil\frac{k}{2}\rceil-1}=\mc{D}_k$.)
Therefore the map $d:E_1^{(2m-1)+2k}(3)\ra E_0^{2m+2k}(3)$ is injective for $k\leq m-1.$
\end{proof}
\begin{lem}\label{lem32}
For $m\leq k\leq 2m-1$, the dimension of the space of cocycles in
$E_1^{(2m-1)+2k}(3)$ is at most one.
\end{lem}
\begin{proof}
We consider the basis of the module $E_{1}^{(2m-1)+2k}(3)$ given by the vectors $\sigma_a$ and $\delta_a$,
where $\lceil\frac{k}{2}\rceil\leq a\leq m$ and the same filtration as defined in the previous lemma. We prove
that the induced map on the hyperplane $\mc{G}_{m-1}$ is injective. Similar computations for the leading terms
in the differentials of the elements of the basis $\{\sigma_*,\delta_*\}$ give that the differential of
$\delta_a$, for $a\neq m$, has the highest term $x^{m-1}\otimes x^{a+1}\otimes x^{k-a}$ (for $\delta_m$ the highest
term is $x^{m}\otimes x^{m}\otimes x^{k-m}$). On the other hand, the leading monomial in the differential of
$\sigma_a$ is $x^{m}\otimes x^{a}\otimes x^{k-a}$. Thus the leading monomials do not interfere except the ones
corresponding to $\delta_m$ and $\sigma_m$ and the differential is injective on the hyperplane
$$\mathcal{G}_{m-1}=\mathcal{D}_m\bigoplus (\mathop{\oplus}\limits_{b\leq m-1}\mathbb{Q}\lan\sigma_{b}\ran).$$
\end{proof}
\begin{rem}
For $k=2m$, $d:E_1^{6m-1}(3)\ra E_0^{6m}(3)$ is injective since
both the spaces are one dimensional and the differential is not
zero. In this case the unique generator is
$\sigma_m=2(x^m\otimes1\otimes x^mG_{12}+x^m\otimes x^m\otimes 1
G_{13}+x^m\otimes x^m\otimes 1G_{23})\in E_1^{6m-1}(3)$. More
generally, the right side of the trapezoid $E_*^*$ is acyclic, see
\cite{AAB}.
\end{rem}
The coordinates of the vector $W$ in the next lemma correspond to the frame
$x^{m}\otimes 1\otimes 1 G_{12}+x^{m}\otimes 1\otimes 1 G_{13}+1\ot x^{m}\otimes 1 G_{23},x^{m-1}\otimes 1
\otimes x G_{12}+x^{m-1}\otimes x\otimes 1 G_{13}+x\ot x^{m-1}\otimes 1 G_{23},\ldots, 1\otimes 1
\otimes x^m G_{12}+1\otimes x^m\otimes 1  G_{13}+ x^{m}\ot1\otimes 1 G_{23}$ of $E_1^{4m-1}(3).$
\begin{lem}\label{lem34}
The kernel of $d:E_{1}^{4m-1}(3)\ra E_{0}^{4m}(3)$ is one dimensional and is generated by the vector
$W=(2m,2m-3,2m-6,\ldots,-m+3,-m)$.
\end{lem}
\begin{proof}
From the previous lemma the dimension of the kernel is $\leq1$,
and now we compute the differential of $W$:
$$ \begin{array}{rl}
  W = & 2m(x^{m}\otimes 1\otimes 1 G_{12}+x^{m}\otimes 1\otimes 1 G_{13}+1\ot x^{m}\otimes 1 G_{23})+ \\
     & +(2m-3)(x^{m-1}\otimes 1\otimes x G_{12}+x^{m-1}\otimes x\otimes 1 G_{13}+x\ot x^{m-1}\otimes 1 G_{23})+\ldots \\
     & +(2m-3a)(x^{m-a}\otimes 1\otimes x^a G_{12}+x^{m-a}\otimes x^a\otimes 1 G_{13}+x^a\ot x^{m-a}\otimes 1 G_{23}) \\
       &+\ldots +(2m-3m)(1\ot 1\ot x^{m} G_{12}+1\ot x^{m}\otimes 1 G_{13}+ x^{m}\otimes 1\ot 1 G_{23}).
\end{array}
$$
A monomial $x^a\ot x^b\ot x^c$, $m\geq a\geq b\geq c\geq 0$,
$a+b+c=2m$ appears at most once in the differential of each of the
rows of $W$, and the total sum of its coefficients equals
$$(2m-3a)+(2m-3b)+(2m-3c)=6m-3(a+b+c)=0.$$
Hence $dW=0$ and the kernel of $d:E_1^{4m-1}(3)\ra E_0^{4m}(3)$ is
one dimensional.
\end{proof}
\begin{lem}\label{lem35}
For $m\leq k\leq 2m-1$, the space of cocycles in $E_1^{(2m-1)+2k}(3)$ is one dimensional.
\end{lem}
\begin{proof}
Multiplying the cocycle $W\in E_1^{(2m-1)+2m}(3)$ by the symmetric
scalar $t_k=x^{k}\otimes 1\otimes1+1\otimes
x^{k}\otimes1+1\otimes1\otimes x^{k}$, $k=0,\ldots,m-1,$ we
obtain an $\mc{S}_3$ invariant cocycle which is not zero, since
the coefficient of $x^{m}\otimes1\otimes x^{k} G_{12}$ in $t_kW$
is $2m+2(2m-3k)=6m-6k$, which is non-zero for $k\neq m$.
\end{proof}
\begin{rem}\label{rem36}
The element $t_{m}W$ is also a cocycle, but $t_{m}W=0.$
\end{rem}
Now we can compute the Poincar\'{e} polynomial for the
unordered configuration space $C(\mathbb{C}P^m,3)$:
\begin{prop}\label{prop37}
The Poincar\'{e} polynomial of the unordered configuration space
$C(\mathbb{C}P^m,3)$ is given by:
$$ \begin{array}{rl}
  P_{C(\mathbb{C}P^m,3)}(t,s)=  &  \mathop{\sum}\limits_{k=0}^{m-1}(P_{3,\leq m}(k))t^{2k} +\mathop{\sum}
  \limits_{k=m}^{2m-1}(P_{3,\leq m}(k)+m-k-1)t^{2k}+  \\
                  & +\mathop{\sum}\limits_{k=2m}^{3m-1}(P_{3,\leq m}(k)-3m+k)t^{2k}+ \mathop{\sum}\limits_{k=2m}^{3m-1}st^{2k-1}= \\
    = &  1+t^2+2t^4+3t^6+4t^8+5t^{10}+7t^{12}+8t^{14}+10t^{16}+\\
     & +12t^{18}+\ldots +2t^{6m-16}+t^{6m-14}+t^{6m-12}+ \\
      & +s(t^{4m-1}+t^{4m+1}+\ldots t^{6m-5}+t^{6m-3}).
\end{array}
$$
where, in the second expansion, the coefficient of $t^{6m-12}$ is stable for $m\geq5$ (the coefficient of $t^{6m-14}$
is stable for $m\geq6$ and the first terms for $m\geq10$).
\end{prop}
\begin{proof}
The proof follows from the computations of the multiplicity of the
trivial representation $V(3)$ in $E_1^{(2m-1)+2k}$ and $E_0^{2k}$
in Section 2 and the Lemmas \ref{lem31} and
\ref{lem35}. For $k=3m-l$, $0\leq l\leq6$, we have $P_{3,\leq m}(k)=P_{3,\leq m}(l)$ so
$$P_{3,\leq m}(k)-3m+k=P_{3,\leq m}(l)-l$$ which is zero for $0\leq l\leq5$; for $l=6$ and $m\geq6$
we have $P_{3,\leq m}(6)=7$, giving $\b_{6m-12}=1.$ For $m=5$ the monomial $t^{18}$
belongs to the second sum and its coefficient is 1.
The initial terms in the second expansion are stable for $m\geq10$: $P_{3,\leq m}(10)=P_3(10)$.
\end{proof}
\begin{rem}
For $m=2,3,4$ we have
\[\begin{array}{cl}
    P_{C(\mathbb{C}P^2,3)}(t,s)= &  1+t^2+t^4+s(t^7+t^9),\\
    P_{C(\mathbb{C}P^3,3)}(t,s)= &  1+t^2+2t^4+2t^6+t^8+s(t^{11}+t^{13}+t^{15}),\\
    P_{C(\mathbb{C}P^4,3)}(t,s)= &  1+t^2+2t^4+3t^6+3t^8+2t^{10}+2t^{12}+s(t^{15}+t^{17}+t^{19}+t^{21}).\\
  \end{array}
\]
\end{rem}
The differentials of the $(2,1)$-component have similar properties.
\begin{lem}\label{lem3.8}
For $0\leq k\leq m-1$, the differential $d_{2,1}$ is injective
$$d:E_{1}^{(2m-1)+2k}(2,1)\rightarrowtail E_{0}^{2m+2k}(2,1).$$
\end{lem}
\begin{proof}
The space $E_{1}^{(2m-1)+2k}(2,1)$ consists of copies of $V(2,1)$, $\mathbf{V}_a$ (where $a$ runs from 0 to $k$),
generated by vectors of the form:

$ \mathbf{V}_a:\left\{\begin{array}{l}
              v_{1,a}=x^{k-a}\ot 1\ot x^a G_{12}-x^{k-a}\ot x^a\ot 1 G_{13} \\
              v_{2,a}=x^{k-a}\ot 1\ot x^a G_{12}-x^{a}\ot x^{k-a}\ot 1 G_{23}.
            \end{array}
\right.$

We define increasing filtrations $\{\mc{V}_i\}$ and $\{\mc{M}_i\}$, $0\leq i\leq\lfloor\frac{k}{2}\rfloor$, on
$E_1^*(2,1)$ and $E_0^*(2,1)$ respectively by
\[\begin{array}{lcl}
    \mc{V}_i & = & \mathop{\bigoplus}\limits_{0\leq a\leq i}(\mathbf{V}_a+\mathbf{V}_{k-a}), \\
    \mc{M}_{-1} & = & (2,1)-\mbox{isotypic component of }\mathop{\mathop\bigoplus
    \limits_{m-1\geq a\geq b\geq c\geq0}}\limits_{a+b+c=m+k}\mathbb{Q}(\mc{S}_3 x^a\o x^b\o x^c), \\
    \mc{M}_i & = & \mc{M}_{i-1}\bigoplus(2,1)-\mbox{isotypic component of }\mathbb{Q}(\mc{S}_3 x^m\o x^{k-i}\o x^i)
  \end{array}
\]
These filtrations are compatible with the symmetric structure and $d(\mc{V}_i)<\mc{M}_i$. We show that
the differential induces an injective map between associated graded $\mc{S}_3$-modules: the matrix of the
induced differential from $\mc{V}_i\diagup\mc{V}_{i-1}$ to $\mc{M}_i\diagup\mc{M}_{i-1}$ contains
the following block given by the coordinates of $dv_{1,i},dv_{2,i},dv_{1,k-i}$ in the canonical basis
$(x^m\o x^{k-i}\o x^i,x^m\o x^{i}\o x^{k-i}, x^{k-i}\o x^m\o x^i,\ldots)$
$$ \left(   \begin{array}{rcr}
                      1 & 1 & 0 \\
                     -1 & 0 & 1 \\
                     1 & 1 & -1 \\
  \end{array}
   \right),
$$
hence the dimension of $d(\mc{V}_i\diagup\mc{V}_{i-1})$ is greater or equal to 3. By Schur lemma $\dim d(\mc{V}_i\diagup\mc{V}_{i-1})=4,$
hence the ``graded" differential is injective. In the very special case of an even $k$, $i=\frac k2$, the component
$\mc{V}_i\diagup\mc{V}_{i-1}\cong \mathbf{V}_i+\mathbf{V}_{k-i}$ has dimension 2, the differential is non-zero, hence injective.
\end{proof}
\begin{lem}\label{lem39}
For $m\leq k\leq 2m$, the dimension of the space of cocycles in the space $E_{1}^{(2m-1)+2k}(2,1)$ is at most two.
\end{lem}
\begin{proof}
In the proof we use the method of Lemmas \ref{lem32} and \ref{lem3.8}.
We will show that the differential is injective on a subspace of codimension two, $\mc{W}$, using an associated graded
morphism. In the case $k=2m$, $\mc{W}=0$ and the injectivity is obvious. Let us define the $V(2,1)$ submodules $\mathbf{W}_a$,
$0\leq a\leq 2m-1-k$, by
$$ \mathbf{W}_a:\left\{\begin{array}{lll}
               w_{1,a} & = & x^{m-a}\ot 1\ot x^{k+a-m} G_{12}-x^{m-a}\ot x^{k+a-m}\ot 1 G_{13} \\
               w_{2,a} & = & x^{m-a}\ot 1\ot x^{k+a-m} G_{12}-x^{k+a-m}\ot x^{m-a}\ot 1 G_{23},
             \end{array}
       \right.
$$
next the increasing filtrations $\mc{W}_i$ and $\mc{N}_i$ for $0\leq i\leq \lfloor\frac{2m-k}{2}\rfloor$ by
\[\begin{array}{lll}
    \mc{W}_0 & =& \mathbf{W}_0,\\
    \mc{W}_i & = & \mathbf{W}_0\oplus (\mathop{\bigoplus}\limits_{1\leq a\leq i}(\mathbf{W}_a+\mathbf{W}_{2m-k-a})), \\
    \mc{W} &  = & \mc{W}_{\lfloor\frac{2m-k}{2}\rfloor},\\
    \mc{N}_{-1} & = & (2,1)-\mbox{isotypic component of }\mathop{\mathop\bigoplus
    \limits_{m\geq a\geq b\geq c\geq0}}\limits_{a+b+c=m+k}\mathbb{Q}(\mc{S}_3 x^a\o x^b\o x^c), \\
    \mc{N}_i & = & \mc{N}_{i-1}\bigoplus(2,1)-\mbox{isotypic component of }\mathbb{Q}(\mc{S}_3 x^m\o x^{m-i}\o x^{k-m+i}).
  \end{array}
\]
The filtrations are compatible with the differential and the symmetric structure. The associated graded differential
$\mc{W}_i\diagup\mc{W}_{i-1}\ra\mc{N}_i\diagup\mc{N}_{i-1}$ is injective:

1) for $i=0$, $\mc{W}_0\diagup\mc{W}_{-1}=\mc{W}_0$ is a simple module and $d(\mc{W}_0)\neq0$;

2) for $i\geq1$, the coordinates of $d(w_{1,i}),d(w_{2,i}),d(w_{1,2m-k-i})$
in the basis $(x^m\o x^{m-i}\o x^{k-m+i},x^m\o x^{k-m+i}\o x^{m-i},x^{m-i}\o x^{m}\o x^{k-m+i},\ldots)$
is given by the matrix in the proof of Lemma \ref{lem3.8};

3) in the special case of an even $k$ and $2i=2m-k$, the component $\mc{W}_i\diagup\mc{W}_{i-1}$ has dimension 2 and the differential
$\mc{W}_i\diagup\mc{W}_{i-1}\ra\mc{N}_i\diagup\mc{N}_{i-1}$ is non-zero, hence injective by Schur lemma.
\end{proof}
\begin{lem}\label{lem310}
For $m\leq k\leq 2m$, the space of cocycles in
$E_1^{(2m-1)+2k}(2,1)$ is two dimensional and it coincides with
the space of coboundaries.
\end{lem}
\begin{proof}
This is a consequence of the previous lemma and of the fact that
the differential $d:E_{2}^{2(2m-1)+2j}(2,1)\cong V(2,1)\ra
E_{1}^{(4m-1)+2j}(2,1)$ is non-zero for $0\leq j\leq m,$ hence the
space of coboundaries has dimension two, by Schur lemma.
\end{proof}
In the next two propositions, the term corresponding to $\lfloor\frac{3m}{2}\rfloor$ and
the term corresponding to $\lceil\frac{3m}{2}\rceil$ should be taken only once for $m$ even.
\begin{prop}\label{prop311}
The Poincar\'{e} polynomial of the $(2,1)$ component of the
cohomology $H^*(F(\mathbb{C}P^m,3))$ is given by:
$$ \begin{array}{lll}
    P_{(F(\mathbb{C}P^m,3))(2,1)}(t,s) & = & \mathop{\sum}\limits_{k=1}^{m-1}2\mu_{2,1}(k)t^{2k}+
    \mathop{\sum}\limits_{k=m}^{\lfloor\frac{3m}{2}\rfloor}2(\mu_{2,1}(k)-k+m-1)t^{2k}+ \\
             & & +\mathop{\sum}\limits_{k=\lceil\frac{3m}{2}\rceil}^{2m-1}2(\mu_{2,1}(3m-k)-k+m-1)t^{2k}+\\
             & & +\mathop{\sum}\limits_{k=2m}^{3m-1}2(\mu_{2,1}(3m-k)-3m+k)t^{2k}=\\
             & = & 2t^2+4t^4+6t^6+10t^8+\ldots +4t^{6m-10}+2t^{6m-8}.
   \end{array}
$$
where the coefficients in the second expansion become stable for $m\geq5.$
\end{prop}
\begin{proof}
This is a direct consequence of dimensions counting for
$E_*^*(2,1)$ in Section 2 and the previous three Lemmas.
The coefficients in the second expansion are computed using Proposition \ref{lem2.1} and the Appendix.
\end{proof}
\begin{rem}
For $m=2,3,4$ we have:
\[\begin{array}{cl}
    P_{(F(\mathbb{C}P^2,3))(2,1)}(t,s)= & 2t^2+2t^4  \\
    P_{(F(\mathbb{C}P^3,3))(2,1)}(t,s)= & 2t^2+4t^4+4t^6+4t^8+2t^{10}\\
    P_{(F(\mathbb{C}P^4,3))(2,1)}(t,s)= & 2t^2+4t^4+6t^6+8t^8+8t^{10}+6t^{12}+4t^{14}+2t^{16}
  \end{array}
\]
\end{rem}
The contribution to the cohomology of the $(1,1,1)$ component is
obvious from the table in Lemma \ref{lem2.1}:
\begin{prop}\label{prop312}
The Poincar\'{e} polynomial for the $(1,1,1)$ component of the
cohomology $H^*(F(\mathbb{C}P^m,3))$ is given by:
$$\begin{array}{lll}
 P_{(F(\mathbb{C}P^m,3))(1,1,1)}(t,s) & = & \mathop{\sum}\limits_{k=3}^{\lfloor\frac{3m}{2}\rfloor}\mu_{1,1,1}(k)t^{2k}+
\mathop{\sum}\limits_{k=\lceil\frac{3m}{2}\rceil}^{3m-3}\mu_{1,1,1}(3m-k)t^{2k}=\\
                                      & = & t^6+2t^8+\ldots +2t^{6m-8}+t^{6m-6}.
\end{array}$$
\end{prop}

\bigskip
\noindent {\em Proof of Theorem \ref{thm1}}. The Poincar\'{e}
polynomial of $E_q^*$ is
$$ P_{E_q^*}(t,s)=\a_q(1+t^2+\ldots+t^{2m})^{3-q}t^{q(2m-1)}s^q, $$
where $\mathop\sum\limits_{i\geq0}\a_it^i=(1+t)(1+2t)$ is the
Poincar\'e polynomial of the Arnold algebra $\mc{A}(3)\cong
H^*(F(\mathbb{C},3))$, see \cite{A}. From the previous Lemmas the Poincar\'e
polynomials of the spaces of cocycles in $E_1^*(3)$ and
$E_1^*(2,1)$ are given by:
$$ \begin{array}{lll}
     P_{Z_1^*(3)}(t,s)   & = & s(t^{4m-1}+t^{4m-3}+\ldots+t^{6m-3}), \\
     P_{Z_1^*(2,1)}(t,s) & = & 2s(t^{4m-1}+t^{4m-3}+\ldots+t^{6m-1}).
   \end{array}
$$
Thus we have
$$ \begin{array}{lll}
    P_{F(\mathbb{C}P^m,3)}(t,s)& = & (1+t^2+t^4+\ldots+ t^{2m})^3-[3(1+t^2+t^4+\ldots+ t^{2m})^2t^{2m}- \\
                               &   & -(t^{4m}+t^{4m+2}+\ldots+ t^{6m-2})-2(t^{4m}+t^{4m+2}+\ldots+ t^{6m})]\\
                               &   & +s(t^{4m-1}+t^{4m+1}+\ldots+ t^{6m-3}),
\end{array}
$$
and a straightforward simplification gives the result.
\hfill$\Box$

In Section \ref{section1} the $(k,q)$ supports of $E_*^*$ and $H_*^*$ are drawn
for each of the three types of $\S3$-modules and for the stable cases: $m\geq5$.

\section{Algebra structure of the cohomology}\label{section4}
The configuration space $F(\mathbb{C}P^m,3)$ is the total space in two natural fibrations:
$$\begin{array}{rllll}
\overset{\circ\circ}{\mathbb{C}P}\,^m  &  \hookrightarrow & F(\mathbb{C}P^m,3) & \ra & F(\mathbb{C}P^m,2),\\
 F(\overset\circ{\mathbb{C}P}\,^m,2) &  \hookrightarrow & F(\mathbb{C}P^m,3) & \ra & \mathbb{C}P^m,
\end{array} $$
see \cite{FN}.
If $X$ is a connected manifold, $\overset\circ{X}$ and $\overset{\circ\circ}{X}$ denote the
spaces $X\setminus\{\mbox{one point}\}$ and $X\setminus\{\mbox{two points}\}$ respectively.
For $m=3$ and $m=4$ the spectral sequences associated to these fibrations are used in
\cite{So} to reduce long computations with the \Kz model. Using the Poincar\'{e} polynomial of
$F(\mathbb{C}P^m,3)$, Theorem \ref{thm1}, we can describe the structure of the Serre spectral sequences
corresponding to these fibrations, proving some conjectures from \cite{So}. The terms of spectral
sequences will be written in bold characters $\mathbf{E}^{*,*}_* $. We will use the notation
$r_m(x,y)=x^m+x^{m-1}y+\ldots+y^m$ and $C_m(t)=r_m(1,t^2)=1+t^2+\ldots+t^{2m}.$
The cohomology of the fibers are given in the next lemma:
\begin{lem}

a) The cohomology algebra of $\overset{\circ\circ}{\mathbb{C}P}\,^m $ has the presentation:
$$ H^*(\overset{\circ\circ}{\mathbb{C}P}\,^m)\cong \mathbb{Q}\lan y,z\mid y^m,yz\ran  $$
with $|y|=2,|z|=2m-1$. Its Poincar\'{e} polynomial is given by
$$ P_{\overset{\circ\circ}{\mathbb{C}P}\,^m }(t)=C_{m-1}(t)+t^{2m-1}.$$
b) The cohomology algebra of $F(\overset\circ{\mathbb{C}P}\,^m,2) $ has the presentation:
$$ H^*(F(\overset\circ{\mathbb{C}P}\,^m,2))\cong \mathbb{C}\lan y,z,u\mid y^m,z^m,
r_m(y,z),yu,zu\ran  $$
with $|y|=|z|=2,|u|=4m-3$. Its Poincar\'{e} polynomial is given by
$$ \begin{array}{lll}
P_{F(\overset\circ{\mathbb{C}P}\,^m,2)}(t) & = & 1+2t^2+3t^4+\ldots +mt^{2m-2}+(m-2)t^{2m}+\\
                                           &   & +(m-3)t^{2m+2}+\ldots +2t^{4m-8}+t^{4m-6}+t^{4m-3}.
\end{array} $$
\end{lem}
\begin{proof}
a) Obviously $\overset\circ{\mathbb{C}P}\,^m\simeq \mathbb{C}P^{m-1} $ and the Mayer-Vietoris sequence for
the decomposition $\overset\circ{\mathbb{C}P}\,^m=\overset{\circ\circ}{\mathbb{C}P}\,^m\cup D^{2m}$ gives
the Betti numbers and the algebra morphism
$H^*(\overset\circ{\mathbb{C}P}\,^m)\to H^*(\overset{\circ\circ}{\mathbb{C}P}\,^m)$, injective for
$*\leq 2m-2$.

b) For these computation we use the version of the \Kz model for a punctured complex projective manifold
$\overset\circ{X}$ introduced in \cite{BMP} ($X$ is simply connected and the coefficients should be complex):
the fundamental class $\omega$ becomes 0.
Denote by $y,z$ and $u$ the cohomology classes of $x\o 1,1\o x$ and $x^{m-1}\o1 G_{12}$ respectively. The equality
$x^m=0$ in the model implies $y^m=z^m=0$ and also $yu=zu=0$; the relation $r_m(y,z)=0$ is a
consequence of $dG_{12}=y^{m-1}z+\ldots+yz^{m-1}$ and of the relations $y^m=z^m=0$. A basis for the cohomology is
$$ \{ 1;y,z;y^2,yz,z^2;...;y^{m-1},y^{m-2}z,...,z^{m-1};y^{m-1}z,...,y^2z^{m-2};...;y^{m-1}z^{m-2};u\}. $$
\end{proof}
We show that the two spectral sequences have a nontrivial differential at a unique place: for the first
fibration, the ``first'' differential, $d_2$, and for the second one, the ``last'' differential,
$d_{2m}$, are non zero.
\begin{prop}\label{prop4.3}
The Serre spectral sequence of the fibration
$$ \overset{\circ\circ}{\mathbb{C}P}\,^m  \hookrightarrow F(\mathbb{C}P^m,3) \ra F(\mathbb{C}P^m,2)$$
collapses at $\mathbf{E}^{*,*}_3$.
\end{prop}
\begin{proof}
The cohomology algebra of the base is a simple consequence of the \Kz model (one can see \cite{So}, Theorem 1,2): its
Poincar\'{e} polynomial is:
$$ P_{F(\mathbb{C}P^m,2)}(t)=C_{m-1}(t)C_m(t).$$
All the differentials in the lower rectangle $(4m-2)\times(2m-2)$ are zero for degree reason. From
\ref{thm1} the $(6m-5)^{th}$ Betti number of $F(\mathbb{C}P^m,3)$ equals 1, therefore
$\mathbf{E}^{4m-4,2m-1}_2$ should contain a cocycle, and this implies that
$$ d_2(1\o z)=(\lambda x\o 1+\mu 1\o x)\o y^{m-1}, $$
where at least one of the complex numbers $\lambda,\mu$ is non zero. As a consequence, the differential
$$ d_2:\mathbf{E}^{i,2m-1}_2\to \mathbf{E}^{i+2,2m-2}_2 $$
\noindent is injective for $i=0,\ldots,2m-2$ and surjective for $i=2m-2,\ldots,4m-2$ because the matrix of
$d_2$ with respect to the bases $\{ x^j\o x^{i-j}\o z\}$ and  $\{ x^j\o x^{i-j}\o y^{m-1}\}$ contains two
triangular blocks, at least one of maximal rank:
$$ \begin{array}{ccc} \left( \begin{array}{lll} \lambda &       & 0 \\
                                 & \dots &   \\
                               * &       & \lambda
\end{array} \right) &  \mbox{ and } &
\left( \begin{array}{lll}  \mu &       & * \\
                             & \dots &   \\
                           0 &       & \mu
   \end{array} \right).
\end{array} 
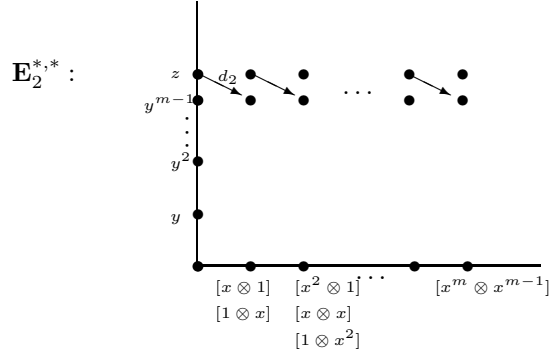
$$

\begin{center}
\begin{picture}(100,120)
\put(-70,80){$\mathbf{E} _2^{*,*}:$}
\put(0,10){ \put(0,0){\line(1,0){130}} \put(0,0){\line(0,1){100}}
 \multiput(-2,-3)(20,0){3}{$\bullet$}  \put(60,-5){$\ldots$}
\multiput(80,-3)(20,0){2}{$\bullet$} \multiput(0,0)(20,0){2}{\multiput(78,60)(0,10){2}{$\bullet$}}
 \multiput(-2,-3)(0,20){3}{$\bullet$} \put(55,65){$\ldots$}
\multiput(-5,43)(0,5){3}{$\cdot$}
\multiput(-2,60)(0,10){2}{$\bullet$} \multiput(-2,60)(20,0){3}{$\bullet$}  \multiput(-2,70)(20,0){3}{$\bullet$}
\begin{scriptsize}\tiny \put(-10,17){$y$} \put(-10,37){$y^2$}  \put(-20,59){$y^{m-1}$} \put(-10,71){$z$}
 \put(7,-10){\tiny$[x\o1]$} \put(7,-20){$[1\o x]$}  \put(37,-10){$[x^2\o1]$}  \put(37,-20){$[x\o x]$}
 \put(37,-30){$[1\o x^2]$} \put(90,-10){$[x^m\o x^{m-1}]$} \put(8,70){$d_2$}
\end{scriptsize}
\multiput(0,73)(20,0){2}{\vector(2,-1){17}}}  \put(79,83){\vector(2,-1){17}}
\end{picture}
\end{center}
\bigskip

\bigskip

Computing the Poincar\'{e} polynomial of the term $\mathbf{E}^{*,*}_3$ we find
$$ \begin{array}{lll}
P_{\mathbf{E}^{*,*}_3}(t) & = & C_{m-2}(t)C_{m-1}(t)C_m(t)+t^{2m-2}C_{m-1}(t)+t^{4m-1}C_{m-1}(t)=\\
                         & = & C_{m-1}^3+t^{4m-1}C_{m-1}(t),
   \end{array}
$$
\noindent the same with $P_{F(\mathbb{C}P^m,3)}(t)$.
\end{proof}
\begin{prop}
The Serre spectral sequence of the fibration
$$ F(\overset\circ{\mathbb{C}P}\,^m,2)\hookrightarrow F(\mathbb{C}P^m,3)\ra \mathbb{C}P^m $$
has a unique non-zero differential $d_{2m}:\mathbf{E}^{0,4m-3}_{2m}\to\mathbf{E}^{2m,2m-2}_{2m} $.
\end{prop}
\begin{proof}
In the proof we will use complex coefficients (necessary for the given presentation of
$H^*(F(\overset\circ{\mathbb{C}P}\,^m,2);\mathbb{C})$); the structure of the differentials in the
spectral sequence with complex coefficients gives the same structure for the rational spectral sequence.

Like in the previous proof, non zero differentials could appear only on the top horizontal line $\mathbf{E}^{*,4m-3}_*$, the
remaining classes having even degrees. Its first Poincar\'{e} polynomial is
$$ P_{\mathbf{E}^{*,4m-3}_2}(t)=t^{4m-3}+t^{4m-1}(1+t^2+t^4+\ldots +t^{2m-2}), $$
whose second part equals the Poincar\'{e} polynomial of the odd cohomology of the space $F(\mathbb{C}P^m,3)$. Therefore
$1\o u\in\mathbf{E}^{0,4m-3}_*$ has a non zero differential and $x^k\o u$ must be a cocycle for $k=1,2,\ldots,m$.

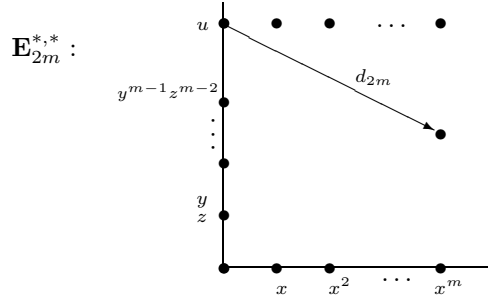
\begin{center}
\begin{picture}(100,120)
\put(0,10){
\put(-80,80){$\mathbf{E}_{2m}^{*,*}:$} \put(0,0){\line(1,0){100}} \put(0,0){\line(0,1){100}}
\multiput(-2,-3)(20,0){3}{$\bullet$}  \put(60,-5){$\ldots$}
\put(80,-3){$\bullet$} \multiput(-2,-3)(0,20){3}{$\bullet$}
\multiput(-5,43)(0,5){3}{$\cdot$}   
\multiput(-2,60)(0,30){2}{$\bullet$} \multiput(-2,90)(20,0){3}{$\bullet$} \put(58,91){$\ldots$}
\put(80,90){$\bullet$}
\begin{scriptsize} \put(-10,24){$y$} \put(-40,64){\tiny$y^{m-1}z^{m-2}$}
\put(-10,17){$z$} \put(-10,90){$u$} \put(20,-10){$x$} \put(40,-10){$x^2$} \put(80,-10){$x^m$}
\put(50,70){$d_{2m}$}
\end{scriptsize}
\put(80,48){$\bullet$}  \put(0,92){\vector(2,-1){80}}
}
\end{picture}
\end{center}

All these imply that
$$ d_2=\ldots=d_{2m-1}=0, d_{2m}(1\o u)=x^m\o (c_0y^{m-1}+c_1y^{m-2}z+\ldots+c_{m-1}z^{m-1})\neq 0,$$
and also that $d_{2m+1}=d_{2m+2}=\ldots =0$.
\end{proof}
In the second part of this section we give a presentation of the cohomology algebras of the
ordered and unordered configuration spaces.
\begin{lem}\label{lem44}
The subalgebra of cocycles $Z_*^*(\mathbb{C}P^m,3)$ in $E_*^*(\mathbb{C}P^m,3)$ has the presentation
(as a graded commutative algebra):
$$\left\lan\begin{array}{l}
             a_1,a_2,a_3,  \\
                            \\
               w,v_1,v_2
                    \end{array}    \right.
\left| \begin{array}{l}
                a_i^{m+1}, s_mw,  (a_i-a_k)v_j,v_1v_2, v_jw, \\
                                                             \\
                                (s_1^2-3s_2)w,(s_1s_2-3s_3)w,
                 \end{array}       \right.
        \left.   \begin{array}{c}      w\left(\begin{array}{c}
                                                            a_1 \\
                                                            a_2 \\
                                                            a_3
                                                          \end{array}\right)
                                                          =A\left(\begin{array}{c}
                                                            w \\
                                                            v_1\\
                                                            v_2
                                                          \end{array}\right)\\
\end{array}\right\ran
$$
where $|a_i|=2$, $|w|=|v_j|=4m-1$, $s_k=a_1^k+a_2^k+a_3^k$ and $A=\left(
                                                                    \begin{array}{ccc}
                                                                      s_1/3 & -a_1 & 2a_1 \\
                                                                      s_1/3 & 2a_1 & -a_1 \\
                                                                      s_1/3 & -a_1 & -a_1 \\
                                                                    \end{array}
                                                                  \right).
$
\end{lem}
\begin{proof}
From Lemmas \ref{lem34}, \ref{lem35} and \ref{lem310} the algebra of cocycles has three generators $x\o1\o1,1\o x\o1,1\o1\o x$
in degree two and three other generators in higher degrees, $W\in E_1^{4m-1}(3)$ and $V_1,V_2\in E_1^{4m-1}(2,1)$,
where
\[\begin{array}{cl}
    V_1= & (x^m\o 1\o 1+x^{m-1}\o x\o1+\ldots+1\o x^m\o 1)G_{13}- \\
         & -(x^m\o1\o1+x^{m-1}\o1\o x+\ldots+1\o1\o x^m)G_{12}, \\
    V_2= & (x^m\o 1\o 1+x^{m-1}\o x\o1+\ldots+1\o x^m\o 1)G_{23}- \\
         & -(x^m\o1\o1+x^{m-1}\o1\o x+\ldots+1\o1\o x^m)G_{12}.
  \end{array}
\]

From the algebra $\mc{B}$ with the above presentation, we define a map $\theta: \mc{B}\ra Z_*^*$ by
$$a_i\mapsto p_i^*(x), w\mapsto W, v_j\mapsto V_j,$$
and this is well defined because:

1) $x^{m+1}=0$ implies $\theta(a_i^{m+1})=0$ and, from Remark \ref{rem36}, $\theta(s_mw)=t_mW=0$;

2)  for every $i=1,2,3,$ the product $p_i^*(x) V_1$ equals
$$(x^m\o x\o1+\ldots+x\o x^{m}\o1)G_{13}-(x^m\o 1\o x+\ldots+x\o1\o x^{m})G_{12};$$

3) the cocycles $WV_i$ and $V_1V_2$ must be zero due to injectivity of $d:E_2^*\ra E_1^*$;

4) the polynomials $t_1^2W$ and $3t_2W$ are symmetric, so it is enough to compute the coefficients of $G_{12}$:
 we find the same value
 $$(18m-36)x^m\o1\o x^2+ (18m-63)x^{m-1}\o1\o x^3+\ldots-(9m-18)x^2\o1\o x^m,$$
 and in a similar way, the coefficients of $G_{12}$ in $t_1t_2W$ and $3t_3W$ coincide:
 $$(18m-54)x^m\o1\o x^3+ (18m-81)x^{m-1}\o1\o x^4+\ldots-(9m-27)x^3\o1\o x^m;$$


5) a basis of the three dimensional space $Z_1^{4m+1}$ is given by the cocycles $t_1W$, $p_1^*(x)V_j$ $(j=1,2)$.
The formulae
\[\begin{array}{cl}
    p_1^*(x)W= & 1/3t_1W-p_1^*(x)V_1+2p_1^*(x)V_2, \\
    p_2^*(x)W= & 1/3t_1W+2p_1^*(x)V_1-p_1^*(x)V_2, \\
    p_3^*(x)W= & 1/3t_1W-p_1^*(x)V_1-p_1^*(x)V_2
  \end{array}
\]
are obtained by direct computations.

The map $\theta$ is surjective because the generators $\{p_i^*(x)\}_{i=1,2,3}$, $\{V_j\}_{j=1,2}$
and $W$ belong to the image. The restriction of $\theta$ to $\mc{B}_0$, the subalgebra generated by
$\{a_i\}_{i=1,2,3}$, is injective. The defining relations of $\mc{B}$ show that each homogenous
component of $\mc{B}_1$, the ideal generated by the elements $w,\{v_j\}_{j=1,2}$ is the span of the elements
$s_1^qw, a_1^qv_1,a_1^qv_2$; therefore $3\geq \dim\mc{B}_1^*\geq\dim Z_1^*=3$
(with one exception: $2\geq \dim\mc{B}_1^{6m-1}\geq\dim Z_1^{6m-1}=2$) and $\theta$ is an isomorphism.
\end{proof}
\bigskip

\noindent {\em Proof of Theorem \ref{thm3}}.
Let us denote by $\mc{A}$ the graded commutative algebra
$$\lan\a_1,\a_2,\a_3,\eta\mid \a_1^{m+1},r_m(\a_i,\a_j), (\a_i-\a_j)\eta,\a_1^m\eta\ran$$
where $|\a_i|=2$ and $|\eta|=4m-1$. We define the algebra morphisms $\varphi, \psi_0$ and $\psi$
\begin{center}
\begin{picture}(50,70)
\put(-30,10){
\put(-40,0){$\mc{A}$} \put(5,0){$H^*(F(\mathbb{C}P^m,3))$}  \put(-30,3){\vector(1,0){30}} \put(-21,6){$\varphi$}
\put(20,40){$Z_*^*(\mathbb{C}P^m,3)$}  \put(120,40){$\mc{B}$}  \put(112,43){\vector(-1,0){36}} \put(90,45){$\theta$} \put(89,35){$\cong$}
\multiput(80,3)(8,0){4}{\line(1,0){4}}  \put(112,3){\vector(1,0){4}}  \put(90,6){$\psi$}
\put(120,0){$\mc{A}$}  \put(128,20){$\psi_0$}
\multiput(55,35)(70,0){2}{\vector(0,-1){25}}  \put(55,19){\vector(0,-1){5}}
}
\end{picture}
\end{center}
and we show that $\varphi$ is surjective and $\psi\varphi=id_{\mc{A}}$.

The morphism $\varphi$ given by
$\a_i\mapsto[p_i^*(x)],\eta\mapsto [W]$ is well defined:

1) $x^{m+1}=0$ implies $\varphi(\a_i^{m+1})=0$;

2) we have $\varphi(r_m(\a_i,\a_j))=[dG_{ij}]$;

3) $\varphi((\a_1-\a_2)\eta)=0$ is a consequence of the following computation
$$\begin{array}{rl}
        (x\o1\o1-1\o x\o1)W= & 3(x^m\o x\o 1+\ldots +x\o x^m\o 1)(G_{23}-G_{13})\\
        = & d(3(1\o x\o1)G_{13}G_{23})
      \end{array}
$$
and similarly for the relation $\varphi((\a_1-\a_3)\eta)=0$;

4) $\varphi(\a_1^m\eta)=[x^m\o1\o1][W]$ and this product is represented by the coboundary
$$\begin{array}{cl}
         (x^m\o1\o1)W & =x^m\o x^m\o1(2m G_{23}-mG_{13})-m(x^m\o1\o x^m) G_{12}= \\
              & =d(m(x^m\o 1\o1)(2G_{12}G_{23}-G_{12}G_{13})).
       \end{array}
$$
The even part of the cohomology is generated by $[p_i^*(x)]_{i=1,2,3}$ and its odd part by $[t_kW]$,
see Lemmas \ref{lem35} and \ref{lem310}, therefore $\varphi$ is surjective.

The algebra morphism $\psi_0:\mc{B}\ra\mc{A}$ given by $a_i\mapsto \a_i, w\mapsto\eta, v_j\mapsto0$ is well defined:

1) $\psi_0(a_1^{m+1})=\a_{1}^{m+1}=0$ and the relations $\psi_0(a_i^{m+1})=0$, $i=2,3$ are consequences of $\a_1^{m+1}=0$
and $(\a_1-\a_i)r_m(\a_1,\a_i)=0;$

2) $\psi_0(s_mw)=0$ is a consequence of $\a_i^m\eta=0$;

3) the relations containing only monomials in $v_1,v_2$ are obviously sent to zero;

4) the relation $(s_1^2-3s_2)w$ is sent to $(\t_1^2-3\t_2)\eta=((3\a_1)^2-9\a_1^2)\eta=0$
(where $\tau_k=\a_1^k+\a_2^k+\a_3^k$), and similarly for $(s_1s_2-3s_3)w$;

5) the last three relations in $\mc{B}$ are sent to zero because $v_j\mapsto 0$ and $\a_1\eta=\frac{1}{3}\t_1\eta.$
\\The morphism $\psi_0$ annihilates the inverse image of the ideal of coboundaries in $Z_*^*$, hence it induces the algebra morphism
$\psi:H^*\ra\mc{A}$:

1) $\psi_0(\theta^{-1}(dG_{ij}))=r_m(\a_i,\a_j)=0;$

2) from Lemma \ref{lem310} we have $\psi_0(\theta^{-1} (dE_2^*))=\psi_0(\theta^{-1}(E_0^*(V_1,V_2)))=0.$
\\Finally, the composition $\psi\varphi$ leaves fixed the generators of $\mc{A}$, hence $\psi\varphi=id.$
\hfill{$\Box$}

In the next lemmas and their proofs the rational polynomials $P_k$ are defined in Section \ref{section1}:
if $T_k=X_1^k+X_2^k+X_3^k$, then $T_k=P_k(T_1,T_2,T_3)$, and $\s_1,\s_2,\s_3$ are the Vi\`{e}te polynomials
expressed in terms of Newton polynomials:
$$\s_1=T_1,\,\,\,\s_2=\frac{1}{2}(T_1^2-T_2),\,\,\,\s_3=\frac{1}{6}T_1^3-\frac{1}{2}T_1T_2+\frac{1}{3}T_3.$$
\begin{lem}\label{lem45}
The algebra of $\mc{S}_3$-invariant cocycles has the presentation:
$$\lan\t_1,\t_2,\t_3,\eta\mid P_{m+1},P_{m+2},P_{m+3},(\t_1^2-3\t_2)\eta, (\t_1\t_2-3\t_3)\eta,P_m\eta\ran,$$
where $|\t_i|=2i$ $(i=1,2,3)$, $|\eta|=4m-1$ and $P_k=P_k(\t_1,\t_2,\t_3)$.
\end{lem}
\begin{proof}
As in the proof of Lemma \ref{lem44} we define a surjective algebra morphism $\theta$ between the algebra $\mc{C}$
given by the above presentation and $Z_*^*(3)$:
$$\theta:\mc{C}\longrightarrow Z_*^*(3),\,\,\,\t_i\mapsto t_i,\, i=1,2,3,\,\,\,\eta\mapsto W.$$
The element $\theta(P_k(\t_1,\t_2,\t_3))=t_k$ is zero in $Z_*^*(3)$ for $k\geq m+1$ and the images of the last three relations
are zero from the steps 1 and 4 in the proof of Lemma \ref{lem44}. The Newton symmetric polynomials $t_1,t_2,t_3$ generate
$Z_0^*(3)$ and, from Lemma \ref{lem35}, $Z_1^*(3)$ is generated by $t_1,t_2,t_3$ and $W$, hence $\theta$ is surjective.
We denote by $\mc{C}_0$ the subalgebra generated by $\t_1,\t_2,\t_3$ and by $\mc{C}_1$ the ideal generated by $\eta$. We have
$$1\geq \dim \mc{C}_1^k\geq\dim Z_1^k(3)=1\,\,\, \mbox{for } k=4m-1,4m+1,\ldots,6m-3$$
due to $Q(\t_1,\t_2,\t_3)\eta=Q(\t_1,\frac{1}{3}\t_1^2,\frac19\t_1^3)\eta;$ this relation and $P_m\eta=0$ imply
$$0=\dim\mc{C}_1^{6m-1}=\dim Z_1^{6m-1}.$$

Now we construct an inverse of the restriction map $\theta|_{\mc{C}_0}$: the algebra map
$$\ve:\mathrm{Symm}[X_1,X_2,X_3]=\mathbb{Q}[T_1,T_2,T_3]\ra \mc{C}_0 \mbox{  given by  } T_i\mapsto\t_i$$
induces a map $Z_0^*(3)\ra\mc{C}_0$ because $\ve$ is zero on the sums
$Q_{a,b,c}=\mathop{\sum}\limits_{\pi\in\mc{S}_3}\pi(X_1^aX_2^bX_3^c)$ $(a\geq b\geq c\geq0)$ if $a\geq m+1:$

1) first $P_{m+1}=P_{m+2}=P_{m+3}=0$ in $\mc{C}$ implies that $P_k=\s_1P_{k-1}-\s_2P_{k-2}+\s_3P_{k-3}$ is also zero
in $\mc{C}$ for any $k\geq m+4;$

2) secondly, $Q_{a,b,0}= P_aP_b-P_{a+b}$ and $$Q_{a,b,c}= P_aP_bP_c-(P_{a+b}P_c+P_{a+c}P_b+P_aP_{b+c}+P_{a+b+c}).$$
\end{proof}
\begin{lem}\label{lem46}
 The vector space $E_{1}^{2k+2m-1}(3)$ is generated by
 $$\G_k=x^k\o1\o1G_{12}+x^k\o1\o1G_{13}+1\o x^k\o1G_{23} \mbox{ and  }M(t_1,t_2,t_3)\G_0,$$
 where $M$ is an arbitrary monomial of degree $2k$.
\end{lem}
\begin{proof}
By induction on $k$: for $k=0$ we have $E_1^{2m-1}(3)=\mathbb{Q}\lan\G_0\ran$. The canonical basis in $E_{1}^{2k+2m-1}(3)$
is given by
$$\{B_{i,j}=x^i\o1\o x^jG_{12}+x^i\o x^j\o1G_{13}+x^j\o x^i\o1G_{23}\}_{i+j=k}.$$
Now we start an induction on $j$: if $(i,j)=(k,0)$ we have $\G_k$; if $j\geq 1$ we find that
$$B_{i,j}=t_1B_{i,j-1}-2B_{i+1,j-1}$$
is a linear combination of $\G_k$ and $M(t_1,t_2,t_3)\G_0$.
\end{proof}
\noindent {\em Proof of Theorem \ref{thm2}}.
This is similar to the proof of Theorem \ref{thm3}; in the next diagram $\mc{D}$
 is the algebra presented in the statement of Theorem \ref{thm2}
and $\mc{C}$ and $\theta$ are from Lemma \ref{lem45}:
\begin{center}
\begin{picture}(50,65)
\put(-70,10){
\put(65,0){$\mc{D}$}
\put(65,40){$\mc{C}$}  \put(120,40){$Z_*^*(3)$}  \put(78,43){\vector(1,0){38}} \put(94,45){$\theta$} \put(92,35){$\cong$}
\multiput(80,3)(8,0){4}{\line(1,0){4}}  \put(112,3){\vector(1,0){4}}  \put(93,6){$\varphi$}
\put(120,0){$H_*^*(3)$}
\multiput(67,35)(70,0){2}{\vector(0,-1){25}}  \multiput(67,19)(70,0){2}{\vector(0,-1){5}}
}
\end{picture}
\end{center}
The induced algebra morphism $\varphi:\t_i\mapsto [t_i],\,i=1,2,3,\,\,\eta\mapsto[W]$ is an isomorphism
because the ideal generated by $\eta$ is isomorphic to $Z_1^*(3)\cong H_1^*(3)$ (see Lemma \ref{lem45}),
and the subalgebra generated by $\t_1,\t_2,\t_3$ is isomorphic to $H_0^*(3)\cong Z_0^*(3)\diagup B_0^*(3)$:
the preimage of coboudaries in $Z_*^*(3)$ is generated by the three relations containing quadratic terms in $P_*$,
as a consequence of Lemma \ref{lem46}
and of the following relations:
\[\begin{array}{rcl}
  d(\G_0) &= & \theta(P_0P_m+P_1P_{m-1}+\ldots+P_mP_0-(m+1)P_m), \\
  d(\G_1) &= & \theta(P_1P_m+P_2P_{m-1}+\ldots+P_mP_1), \\
  d(\G_2) &= & \theta(P_2P_m+P_3P_{m-1}+\ldots+P_mP_2),\\
  d(M(t_1,t_2,t_3)\G_0) & = & M(t_1,t_2,t_3)d(\G_0),\\
    d(\G_k)&= &  d(\s_1\G_{k-1}-\s_2\G_{k-2}+\s_3\G_{k-3})= \\
           &=& \s_1d(\G_{k-1})-\s_2d(\G_{k-2})+\s_3d(\G_{k-3})).
\end{array}\]
For $k=4$, in the last relations we have to use also the relation $\s_3P_m=0$,
and this is a consequence of $P_k=0$ for $k=m+1,m+2,m+3$.
\hfill{$\Box$}

\section{The stable cohomology classes}
For the infinite dimensional complex projective space we will consider configurations with an arbitrary number of points.
\begin{theorem}\label{theorem5.1}
The inclusion map
$$ \iota:F(\mathbb{C}P^{\infty},n)\hookrightarrow (\mathbb{C}P^{\infty})^n $$
is a homotopy equivalence.
\end{theorem}
\begin{proof}
We define the homotopy inverse of the inclusion map $\iota$ by
$$ f:(\mathbb{C}P^{\infty})^n\to F(\mathbb{C}P^{\infty},n), (P_0,P_1,\ldots,P_{n-1})\mapsto (Q_0,Q_1,\ldots,Q_{n-1}), $$
where the coordinates of $Q_j=[q^j_0:q^j_1:\dots]$ depend only on the coordinates of $P_j=[p^j_0:p^j_1:\dots]$:
$$ q^j_{ni+j}=p^j_i \mbox{ and } q^j_i=0 \mbox{ for the rest of coordinates.} $$
For instance, if $n=3$, the triple $(P_0,P_1,P_2)$ (of not necessarily distinct points)
is sent to the triple of distinct points $(Q_0,Q_1,Q_2)$:
\[\begin{array}{ccc}
  P_0=[p_0^0:p_1^0:p_2^0:\ldots] &         & Q_0=[p_0^0:0:0:p_1^0:0:0:p_2^0:0:0:\ldots] \\
  P_1=[p_0^1:p_1^1:p_2^1:\ldots] & \mapsto & Q_1=[0:p_0^1:0:0:p_1^1:0:0:p_2^1:0:\ldots] \\
  P_2=[p_0^2:p_1^2:p_2^2:\ldots] &         & Q_2=[0:0:p_0^2:0:0:p_1^2:0:0:p_2^2:\ldots]
\end{array}
\]
The homotopy is given by $H:(\mathbb{C}P^\infty)^n\times I\ra (\mathbb{C}P^\infty)^n$,
$$((P_0,\ldots,P_{n-1}),t)\mapsto ((1-t)P_0+tQ_0,\ldots,(1-t)P_{n-1}+tQ_{n-1}),$$
where, for $P=[p_0:p_1:\ldots]\in\mathbb{C}P^\infty$ and $Q=[q_0:q_1:\ldots]$ with $q_i=$ linear form in $p_0,p_1,\ldots,$
the point $(1-t)P+tQ$ is defined as $[(1-t)p_0+tq_0:(1-t)p_1+tq_1:\ldots]$; this makes sense in our case
because the last non-zero entry $p_i$ gives a last non-zero entry in $(1-t)P+tQ$. The map $H$ is continuous on finite skeleta
of $(\mathbb{C}P^\infty)^n\times I$ and its restriction $h$ to $F(\mathbb{C}P^{\infty},n)\times I$ takes values in $F(\mathbb{C}P^{\infty},n)$:
in fact we have $H((\mathbb{C}P^{\infty})^n\times(0,1])\subset F(\mathbb{C}P^{\infty},n)$.
Obviously $H\mid_{t=0}=id=h\mid_{t=0}$, $H\mid_{t=1}=\iota\circ f, h\mid_{t=1}=f\circ \iota.$
\end{proof}
\begin{cor}\label{cor5.2}
The induced action of the symmetric group $ \mathcal{S}_n$ on the cohomology algebra of the configuration
space $F(\mathbb{C}P^{\infty},n)$ is the natural action on the polynomial ring $\mathbb{Q}[x_1,\dots x_n]$.
\end{cor}
\begin{proof}
The inclusion $ F(\mathbb{C}P^{\infty},n)\hookrightarrow (\mathbb{C}P^{\infty})^n $ is $\mathcal{S}_n$-equivariant
and induces an isomorphism in cohomology.
\end{proof}
\noindent {\em Proof of Theorem \ref{thm4}. }Clear from \ref{theorem5.1} and \ref{cor5.2}. It is also clear that the
first isomorphism is true with coefficients in $\mathbb{Z}$.
\hfill$\Box$

\bigskip

\noindent {\em Proof of Corollary \ref{cor2}. }Using \cite{T}, the image of the map
$i^*:H^*(X^n)\to H^*(F(X,n))$ is contained in $H^*_0(F(X,n))$. For the space $\mathbb{C}P^m$, the
first non zero coboundaries are $d(E^{2m-1}_1)< E^{2m}_0$. In the range $k\in \{0,1,\dots,2m-1\}$ we have
$$ \begin{array}{lll}
   H_*^k(F(\mathbb{C}P^m,n)) & =     & H_0^k(F(\mathbb{C}P^m,n))\cong E_0^k(F(\mathbb{C}P^m,n))\cong H^k((\mathbb{C}P^m)^n)\\
                             & \cong & H^k((\mathbb{C}P^{\infty})^n)\cong H^k(F(\mathbb{C}P^{\infty},n)).
   \end{array}
$$
\hfill$\Box$


\section{Configurations of collinear and non collinear points}

For dimensions $2\leq m\leq\infty$ the 3-point configuration space $F(\mathbb{C}P^m,3)$ splits into two parts:
$$F(\mathbb{C}P^m,3)=F_{c}(\mathbb{C}P^m,3)\bigsqcup F_{nc}(\mathbb{C}P^m,3),$$
where $F_c(\mathbb{C}P^m,3)$ is the space of all configurations of
three collinear points in $\mathbb{C}P^m$ and
$F_{nc}(\mathbb{C}P^m,3)$ is the configuration space of three non-collinear points.
We compute the Betti numbers and (partially) the multiplicative structure of the cohomology algebras of these spaces.
We begin with the finite dimensional case.
\begin{prop}\label{prop6.1}
The Poincar\'{e} polynomial of the space of configurations of three collinear points in $\mathbb{C}P^m$ is given by:
$$P_{F_c(\mathbb{C}P^m,3)}(t)=(1+t^{2m+1})C_{m-1}(t).$$
\end{prop}
\begin{proof}
There are two natural fibrations (see \cite{FN} and \cite{BP}, Proposition 2.2):
$$\begin{array}{lc}
(I) & S^1\simeq \overset{\circ\circ}{\mathbb{C}P}\,^1\hookrightarrow
F_c(\mathbb{C}P^m,3)\mathop{\ra}\limits^p F(\mathbb{C}P^m,2)\\
(II) & F(\mathbb{C}P^1,3)\hookrightarrow F_c(\mathbb{C}P^m,3)\ra Gr_1(\mathbb{C}P^m)
\end{array}$$
The bases of these fibrations are simply connected and the
$\mathbf{E}_2$ page of the Leray-Serre spectral sequence are
given in the figure:
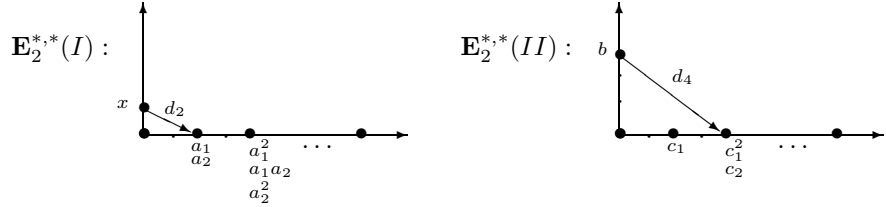
\begin{center}
\begin{picture}(200,85)
\put(-60,60){$\mathbf{E}_2^{*,*}(I):$}
\put(-10,30){
\put(0,0){\vector(1,0){100}}
\put(0,0){\vector(0,1){50}}
\multiput(-2,-2)(20,0){3}{$\bullet$} \multiput(10,-3)(20,0){2}{$\cdot$}  \put(60,-5){$\ldots$}  \put(80,-2){$\bullet$}
\multiput(-2,-2)(0,10){2}{$\bullet$}   \put(0,10){\vector(2,-1){18}}
\begin{scriptsize}
\put(-10,10){$x$} \put(18,-6){$a_1$} \put(18,-11){$a_2$}  \put(40,-8){$a_1^2$}
\put(40,-15){$a_1a_2$} \put(40,-24){$a_2^2$} \put(8,8){$d_2$}
\end{scriptsize}
\put(180,0){
\put(-60,30){$\mathbf{E}_2^{*,*}(II):$}
\put(0,0){\vector(1,0){100}}
\put(0,0){\vector(0,1){50}}
\multiput(-2,-2)(20,0){3}{$\bullet$} \multiput(10,-3)(20,0){2}{$\cdot$}  \put(60,-5){$\ldots$}  \put(80,-2){$\bullet$}
\multiput(-2,-2)(0,30){2}{$\bullet$} \multiput(-1,10)(0,10){2}{$\cdot$}   \put(0,30){\vector(4,-3){38}}
\begin{scriptsize}
\put(-8,30){$b$} \put(18,-6){$c_1$}  \put(40,-8){$c_1^2$}
\put(40,-15){$c_2$}  \put(20,20){$d_4$}
\end{scriptsize}
}}
\end{picture}
\end{center}

The cohomology algebra $H^+(F(\mathbb{C}P^1,3))$ is concentrated in degree three
and the odd cohomology of the Grassmanian $Gr_1(\mathbb{C}P^m)=Gr_2(\mathbb{C}^{m+1})$ is zero, hence $\b_1(F_c(\mathbb{C}P^m,3))=0.$
This implies that $d_2(x)=\l a_1+\mu a_2$ in $\mathbf{E}_2^{*,*}(I)$,
where at least one of the scalars $\l$ and $\mu$ is non zero. By the same argument as that in Proposition \ref{prop4.3}
we see that $d_2:\mathbf{E}_2^{i,1}\ra \mathbf{E}_2^{i+2,0}$ is injective up to $i=2m-2$ and
surjective for $i=2m-2,\ldots,4m-2$. Thus the Poincar\'{e} polynomial is
$$P_{F_c(\mathbb{C}P^m,3)}(t)=(1+t^{2m+1})C_{m-1}(t).$$
\end{proof}
\begin{prop}
The Poincar\'{e} polynomial of $F_{nc}(\mathbb{C}P^m,3)$ is given by:
$$P_{F_{nc}(\mathbb{C}P^m,3)}(t)=C_{m-2}(t)\cdot C_{m-1}(t)\cdot C_m(t).$$
\end{prop}
\begin{proof}
By Poincar\'e-Lefschetz duality we have $$H_i(\mathbb{C}P^m\setminus\mathbb{C}P^1)\cong H^{2m-i}(\mathbb{C}P^m,\mathbb{C}P^1);$$
from the cohomology long exact sequence of the pair $(\mathbb{C}P^m,\mathbb{C}P^1)$ we find
$$P_{\mathbb{C}P^m\setminus\mathbb{C}P^1}(t)=C_{m-2}(t).$$
Therefore the spectral sequence associated to the fibration (see \cite{BP}, Remark 2.5)
$$\mathbb{C}P^m\setminus\mathbb{C}P^1\hookrightarrow F_{nc}(\mathbb{C}P^m,3)\ra F(\mathbb{C}P^m,2)$$
is concentrated in even degrees and degenerates at the $\mathbf{E}_2$ term; the Poincar\'{e} polynomial is given by
$$P_{F_{nc}(\mathbb{C}P^m,3)}=C_{m-2}(t)\cdot C_{m-1}(t)\cdot C_m(t).$$
\end{proof}
In the second part we compute the cohomology algebras for collinear and non-collinear configurations in
the infinite dimensional projective space.
\begin{prop}\label{prop6.3}
The configuration space of three collinear points in $\mathbb{C}P^\infty$ has the cohomology algebra
of $\mathbb{C}P^\infty$ $$H^*(F_{c}(\mathbb{C}P^{\infty},3))\cong H^*(\mathbb{C}P^\infty) .$$
\end{prop}
\begin{proof}
We will use the spectral sequence argument given in Proposition \ref{prop6.1}, first proving that
the following are Serre fibrations:
\[\begin{array}{cc}
  (I) & S^1\simeq \overset{\circ\circ}{\mathbb{C}P}\,^1\hookrightarrow F_c(\mathbb{C}P^\infty,3)
  \mathop{\ra}\limits^p F(\mathbb{C}P^\infty,2)\simeq(\mathbb{C}P^{\infty})^2 \\
  (II) & F(\mathbb{C}P^1,3)\hookrightarrow F_c(\mathbb{C}P^\infty,3)\ra Gr_1(\mathbb{C}P^\infty).
\end{array}\]
For the second one the image of $H:e^m\times I\ra Gr_1(\mathbb{C}P^\infty)$ is contained in some finite skeleton $Gr_1(\mathbb{C}P^r)$.
Also there is an $s$ such that $h(e^m)\subset(\mathbb{C}P^s)^3$ so $\mathrm{Im} (h)$ lies in
$F_c(\mathbb{C}P^\infty,3)\bigcap(\mathbb{C}P^s)^3=F_c(\mathbb{C}P^s,3)$:
\begin{center}
\begin{picture}(50,50)
\put(-30,0){
\put(-20,0){$e^m\times I\mathop{\longrightarrow}\limits^{H}Gr_1(\mathbb{C}P^m)$}
\put(-10,40){$e^m\mathop{\longrightarrow}\limits^{h}\,\,F_c(\mathbb{C}P^m,3)$}
\put(80,0){$\longrightarrow Gr_1(\mathbb{C}P^\infty)$}
\put(80,40){$\longrightarrow F_c(\mathbb{C}P^\infty,3)$}
\multiput(-5,35)(60,0){3}{\vector(0,-1){25}}
\put(5,10){\vector(3,2){35}}
}
\end{picture}
\end{center}
Take $m$ to be the maximum of $r$ and $s$, then the homotopy $H$ lifts to $F_c(\mathbb{C}P^m,3)$ because
$F(\mathbb{C}P^1,3)\hookrightarrow F_c(\mathbb{C}P^m,3)\ra Gr_1(\mathbb{C}P^m)$ is a locally trivial fibration (see \cite{BP}).
Similarly one can show that (II) is also a Serre fibration.
The spectral sequences associated to these fibrations start with:
\begin{center}
\begin{picture}(200,85)
\put(-60,60){$\mathbf{E}_2^{*,*}(I):$}
\put(-10,30){
\put(0,0){\vector(1,0){100}}
\put(0,0){\vector(0,1){50}}
\multiput(-2,-2)(20,0){3}{$\bullet$} \multiput(10,-3)(20,0){2}{$\cdot$}  \put(60,-5){$\ldots$}  \put(80,-2){$\bullet$}
\multiput(-2,-2)(0,10){2}{$\bullet$}   \put(0,10){\vector(2,-1){18}}
\begin{scriptsize}
\put(-10,10){$x$} \put(18,-6){$y$} \put(18,-11){$z$}  \put(40,-8){$y^2$}
\put(40,-15){$yz$} \put(40,-22){$z^2$} \put(8,8){$d_2$}
\end{scriptsize}
\put(180,0){
\put(-60,30){$\mathbf{E}_2^{*,*}(II):$}
\put(0,0){\vector(1,0){100}}
\put(0,0){\vector(0,1){50}}
\multiput(-2,-2)(20,0){3}{$\bullet$} \multiput(10,-3)(20,0){2}{$\cdot$}  \put(60,-5){$\ldots$}  \put(80,-2){$\bullet$}
\multiput(-2,-2)(0,30){2}{$\bullet$} \multiput(-1,10)(0,10){2}{$\cdot$}   \put(0,30){\vector(4,-3){38}}
\begin{scriptsize}
\put(-8,30){$b$} \put(18,-6){$c_1$}  \put(40,-8){$c_1^2$}
\put(40,-15){$c_2$}  \put(20,20){$d_4$}
\end{scriptsize}
}}
\end{picture}
\end{center}
From the second spectral sequence, the first Betti number of $F_c(\mathbb{C}P^\infty,3)$ is zero,
hence in the first spectral sequence we have $d_2(x)=y$, after a change of basis. Therefore, in the first spectral sequence,
$$\mathbf{E}_3\cong\mathbf{E}_\infty\cong\mathbb{C}[z]   \mbox{    where   } |z|=2.$$
\end{proof}
\begin{rem}
Using the ideas from Theorem \ref{theorem5.1} one can find a map $f$ $$\mathbb{C}P^\infty\mathop{\ra}\limits^f F_c(\mathbb{C}P^\infty,3)\mathop{\ra}\limits^{pr_1} \mathbb{C}P^\infty,\,\, P=[p_0:p_1:\ldots]\mapsto(Q_0,Q_1,Q_2),$$
where $Q_0=[p_0:0:p_1:0:\ldots], Q_1=[0:p_0:0:p_1:\ldots],Q_2=[p_0:p_0:p_1:p_1:\ldots],$
such that $pr_1\circ f\simeq id$, hence the algebraic isomorphism from Proposition \ref{prop6.3} is induced by a continuous map.
\end{rem}
\begin{prop}
The configuration space of three non-collinear points in $\mathbb{C}P^\infty$ has the homotopy type
of the ordered configuration space of three points
$$F_{nc}(\mathbb{C}P^{\infty},3)\simeq F(\mathbb{C}P^{\infty},3)\simeq (\mathbb{C}P^{\infty})^{3}.$$
\end{prop}
\begin{proof}
In the proof of Theorem \ref{theorem5.1}, the points $Q_0,Q_1,Q_2$ are non-collinear.
\end{proof}
\begin{rem}
More generally, $F_{nc}(\mathbb{C}P^{\infty},n)\simeq (\mathbb{C}P^{\infty})^{n}$, with the same proof.
\end{rem}


\section{Appendix: Counting partitions in three parts}

In this elementary section we give closed formulae for the number of partitions
\[\begin{array}{rl}
    P_3(k)= & card\{(a,b,c)\mid  a\geq b\geq c\geq0,\,\,\, a+b+c=k\}, \\
    P_{3,\leq m}(k)= & card\{(a,b,c)\mid m\geq a\geq b\geq c\geq 0,\,\,\, a+b+c=k\}
  \end{array}
\]
and we prove unimodality for the sequences of even and odd Betti
numbers for the configuration spaces discussed in this paper. We
will denote by $\lfloor x\rceil$ the nearest integer to the real
number $x$; for $x\in\mathbb{Z}+\frac{1}{2}$ this is not defined
but in all the following formulae this never happens.
\begin{prop}
a) The number of positive partitions of $k$ into three parts
$a+b+c=k,$ $ a\geq b\geq c\geq1$ is given by $\lfloor\frac{k^2}{12}\rceil$.

b) The number of non-negative partitions of $k$ is given by $P_3(k)=\lfloor\frac{(k+3)^2}{12}\rceil$
\end{prop}
\begin{proof}
a) An expanded version for this (six cases) is given in \cite{H}. The closed form can be found in \cite{YY}.

b) The non-negative partitions $(a,b,c)$ of $k$ are in bijection with positive partitions $(a+1,b+1,c+1)$ of $k+3$.
\end{proof}
\begin{prop}
The number of bounded partitions $P_{3,\leq m}(k)$ satisfy the following properties:

a) For $0\leq k\leq m$, $P_{3,\leq m}(k)=P_{3}(k)=\lfloor\frac{(k+3)^2}{12}\rceil$.

b) For $k\geq 3m+1$, $P_{3,\leq m}(k)=0$.

c) For $\frac{3m}{2}\leq k\leq 3m$, $P_{3,\leq m}(k)=P_{3,\leq m}(3m-k).$
\end{prop}
\begin{proof}
a) and b) are obvious, and for c) use the bijection $$(a,b,c)\leftrightarrow (m-c,m-b,m-a).$$
\end{proof}
From this result, it is sufficient to compute the value of $P_{3,\leq m}(k)$ for $k$ between
$m+1$ and $\lfloor\frac{3m}{2}\rfloor$.
\begin{prop}
For $m+1\leq k\leq \lfloor\frac{3m}{2}\rfloor$, the number of bounded partitions is given by
$$P_{3,\leq m}(k)=\big\lfloor\frac{(k+3)^2}{12}\big\rceil-\big\lfloor \frac{(k-m-1)^2}{4}\big\rceil+m-k.$$
\end{prop}

\begin{proof}
From the total number of non-negative partitions, $P_{3}(k)$, the number of solutions of $a+b+c=k$ satisfying
$\max(a,b,c)\geq m+1$ should be subtracted. Because $a\geq b\geq c$ and $k\leq \lfloor \frac{3m}{2}\rfloor$, only
$a$ could be greater than $m$: $a=m+1,m+2,\ldots,k$. For a fixed $a$, the number of solutions of
$$b+c=k-a,\,\,\,\,\, b\geq c\geq 0$$
equals $\lfloor\frac{k-a}{2}\rfloor+1$. A simple computation gives
$\lfloor\frac{0}{2}\rfloor+\lfloor\frac{1}{2}\rfloor+\ldots+\lfloor\frac{p}{2}\rfloor=\lfloor\frac{p^2}{4}\rfloor,$
and taking $p=k-m-1$, we obtain the formula.
\end{proof}
Now we show that the sequences of even and odd Betti numbers of the cohomology algebra
$H^*(C(\mathbb{C}P^m,3))\cong H^*(F(\mathbb{C}P^m,3))(3)$ are unimodal.
\begin{lem}

a) The sequence $(P_3(k))_{k\geq 0}$ is increasing (strictly for $k\geq1$).

b) The sequence $(P_{3,\leq m}(k))_{0\leq k\leq 3m}$ is unimodal and symmetric, more precisely:
$\begin{array}{c}
   1=P_{3,\leq m}(0)=P_{3,\leq m}(1)<P_{3,\leq m}(2)<\ldots<P_{3,\leq m}(\lfloor\frac{3m}{2}\rfloor)= \\
   =P_{3,\leq m}(\lceil\frac{3m}{2}\rceil)>\ldots>P_{3,\leq m}(3m-1)=P_{3,\leq m}(3m)=1
 \end{array}$
\end{lem}
\begin{proof}

a) A solution of the equation $a+b+c=k$ gives the solution $(a+1)+b+c=k+1$ and solutions of the form
$a+a+b=k+1$ $(a\geq b)$ can not be obtained in this way (for $k\neq 1$ one can find solutions of the form $a+a+b=k+1$).

b) Similarly, for $k\leq\lfloor\frac{3m}{2}\rfloor$, a solution $a+b+c=k,\,\,m-1\geq a\geq b\geq c$, gives
a solution $(a+1)+b+c=k+1,\,\,m\geq a+1\geq b\geq c$, and a solution of the form $m+b+c=k$ gives a solution
$m+(b+1)+c=k+1$ ($b$ can not be equal to $m$ for $k\leq\lfloor\frac{m}{2} \rfloor$). Again we have strict
inequality for $k\geq 1$ because solutions of the form $a+a+b=k,\,\,(a\geq b),$ are not in the image of the previous injective map.
\end{proof}

\noindent {\em Proof of Corollary 1.2.} Case 1: Betti numbers of $F(\mathbb{C}P^m,3)$. The even Poincar\'{e} polynomial equals
$(1+t^2+\ldots +t^{2m-1})^3$ and the odd Poincar\'{e} polynomial equals $t^{4m-1}(1+t^2+t^4+\ldots+ t^{2m-2})$.

Case 2: $V(3)$-Betti numbers. The sequence of odd Betti numbers is $\b_1=\b_3=\ldots=\b_{4m-3}=0$, $\b_{4m-1}=\b_{4m+1}=\ldots=\b_{6m-3}=1$.
The sequence of even Betti numbers can be splitted into four parts:

i) If $0\leq k\leq m-1$, $\b_{2k}=P_{3,\leq m}(k)$: the sequence is increasing (strictly increasing for $k\geq 1$);

ii) If $m\leq k\leq\lfloor\frac{3m}{2}\rfloor$, $\b_{2k}=P_{3,\leq m}(k)+m-k-1$: the sequence is increasing because
$P_{3,\leq m}(k)<P_{3,\leq m}(k+1)$ implies
$$P_{3,\leq m}(k)+m-k-1\leq P_{3,\leq m}(k+1)+m-k-2;$$

iii) If $\lceil\frac{3m}{2}\rceil\leq k\leq 2m-1$,
$\b_{2k}=P_{3,\leq m}(k)-3m+k$: the sequence is the sum of two
strictly decreasing sequences;

iv) If $2m\leq k\leq 3m-1$, $\b_{2k}=P_{3,\leq m}-3m+k$: the sequence is decreasing as in the second case.

Finally, the join of these sequences is unimodal:

a) $\b_{2m-2}\leq\b_{2m}$, because $P_{3,\leq m}(m-1)\leq P_{3,\leq m}(m)-1$ (here we need $m\geq2$)

b) $\b_{2\lfloor\frac{3m}{2}\rfloor}\geq \b_{2\lceil\frac{3m}{2}\rceil}$; $P_{3,\leq m}(\lfloor\frac{3m}{2}
\rfloor)+m-\lfloor\frac{3m}{2}\rfloor-1\geq P_{3,\leq m}(\lceil\frac{3m}{2}\rceil)+m-\lceil\frac{3m}{2}\rceil-1;$

c) $\b_{4m-2}\geq\b_{4m}$, because $P_{3,\leq m}(2m-1)-m\geq P_{3,\leq m}(2m)-m$.

Case 3: $V(2,1)$-Betti numbers. Like in the previous case, the terms of the first two sums in the Proposition \ref{prop311}
are given by an increasing sequence and those of the last two sums are given by a decreasing sequence.

Case 4: $V(1,1,1)$-Betti numbers. Using Proposition \ref{prop312} this is obvious: the even sequence starts
with three zeros, next is the increasing sequence $P_{3,\leq m-2}(k-3)$, $3\leq k\leq\lfloor\frac{3m}{2}\rfloor$,
next the symmetric decreasing sequence, at the end there are three zeros and the odd sequence is constant zero.
\hfill{$\Box$}
\begin{rems}
a) In the special case of the projective line $(m=1)$, the
sequences of Betti numbers are still unimodal:
$$P_{C(\mathbb{C}P^1,3)}(t)=1+t^3.$$

b) The sequences of even Betti numbers is not symmetric, for instance
$$P_{C(\mathbb{C}P^2,3)}(t)=(1+t^2+2t^4+t^6)+(t^7+t^9).$$
\end{rems}


\begin{thebibliography}{6}
\bibitem[A] {A}  V.I. Arnold, \emph{The cohomology ring of dyed braids}, Mat. Zametki 5 (1969), 227-231.

\bibitem[AAB] {AAB} S. Ashraf, H. Azam, B. Berceanu, \emph{Representation theory for the Kri\v{z} model},
arXiv: 1106.4926v1 [math.RT] (2012).

\bibitem[BMP] {BMP} B. Berceanu, M. Markl, S. Papadima, \emph{Multiplicative models for configuration spaces of
algebraic varieties}, Topology 44 (2005), 415-440.

\bibitem[BP] {BP} B. Berceanu, S. Parveen, \emph{Braid groups in complex projective spaces}, Advances in Geometry 12 (2012), 269-286.

\bibitem[B] {B} R. Bezrukavnikov, \emph{Koszul dg-algebras arising from configuration spaces}, Geometric and
Functional Analysis 4(2) (1994), 119-135.

\bibitem[CT]{CT} F. Cohen, L. Taylor, \emph{Computations of Gelfand-Fuks cohomology,
the cohomology of function spaces, and the cohomology of configuration spaces}; in: Geometric Applications of Homotopy Theory I,
Proceedings, Evanston 1977, Lecture Notes in Mathematics vol. 657, 106-143, Springer-Verlag 1978.

\bibitem[FaH] {FaH} E.R. Fadell, S.Y. Husseini, \emph{Geometry and Topology of Configuration Spaces}, Springer
Monographs in Mathematics, Springer-Verlag, Berlin, 2001.

\bibitem[FN] {FN} E. Fadell, L. Neuwirth, \emph{Configuration spaces}, Math. Scand. 10 (1962), 111-118.

\bibitem[FZ] {FZ} E. M. Feichtner, G. M. Ziegler, \emph{The integral cohomology algebras of ordered configuration spaces of spheres},
 Documenta Mathematica 5 (2000), 115-139.

\bibitem[FOT] {FOT} Y. Felix, J. Oprea, D. Tanr\'{e}, \emph{Algebraic Models in Geometry}, Oxford Graduate Texts in Mathematics 17, 2008.

\bibitem[FT] {FT} Y. Felix, D. Tanr\'{e}, \emph{The cohomology algebra of unordered configuration spaces}, J. of
London Math. Soc. 72(2) (2005), 525-544.

\bibitem[FTh] {FTh} Y. Felix, J. C. Thomas, \emph{Rational Betti numbers of configuration spaces}, Topology Appl. 102(2) (2000), 139-149.


\bibitem[FH] {FH} W. Fulton, J. Harris, \emph{A First Course in Representation Theory}, Graduate Texts in
Mathematics 129, Springer-Verlag Berlin, 1991.

\bibitem[FM] {FM} W. Fulton, R. MacPherson, \emph{A compactification of configuration spaces}, Ann. of Math. 139 (1994), 183-225.

\bibitem[H] {H} M. Hall, Jr., \emph{Combinatorial Theory}, Wiley Interscience, 1998.

\bibitem[K] {K} I. Kriz, \emph{On the rational homotopy type of configuration spaces}, Ann. of Math. 139 (1994), 227-237.

\bibitem[LS] {LS} G. I. Lehrer, L. Solomon, \emph{On the action of the symmetric group on the cohomology of the complement of its reflecting hyperplanes}, J. Algebra 104 (1986), 410-424.

\bibitem[S] {So} T. Sohail, \emph{Cohomology of configuration spaces of complex projective spaces}, Czechoslovak Math.
Journal vol. 60(2) (2010), 411-422.

\bibitem[T] {T} B. Totaro, \emph{Configuration spaces of algebraic varieties}, Topology 35(4) (1996), 1057-1067.

\bibitem[YY] {YY} A. M. Yaglom, I. M. Yaglom, \emph{Challenging Mathematical Problems with Elementary Solutions},
Holden-Day, Inc., San Francisco, CA, 1967.
\end{thebibliography}
\end{document}